\documentclass[]{amsart}
\usepackage{latexsym,amssymb,amsmath,amsthm}
\usepackage{comment}
\usepackage{float}
\usepackage{graphicx}
\newtheorem{thm}{Theorem}
\newtheorem{lemma}[thm]{Lemma}
\newtheorem{propo}[thm]{Proposition}
\newtheorem{coro}[thm]{Corollary}
\theoremstyle{definition}
\newtheorem{defn}{Definition}

\def\G{\mathcal{G}}
\def\tr{\rm{tr}}
\def\ind{\rm{ind}}
\def\dim{\rm{dim}}
\def\ker{\rm{ker}}
\def\im{\rm{im}}
\def\deg{{\rm deg}}

\def\coker{\rm{coker}}
\def\gcat{\rm{gcat}}
\def\Cat{\rm{Cat}}
\def\deg{\rm{deg}}
\def\str{\rm{str}}

\title{The McKean-Singer Formula in Graph Theory}
\author{Oliver Knill}
\date{January 6, 2013}
\address{
        Department of Mathematics \\
        Harvard University \\
        Cambridge, MA, 02138
        }
\subjclass{Primary: 05C50,81Q10 }
\keywords{Graph theory, Heat kernel, Dirac operator, Supersymmetry}

\begin{document}
\maketitle
\begin{abstract}
For any finite simple graph $G=(V,E)$, the discrete Dirac operator $D=d+d^*$ and the 
Laplace-Beltrami operator $L=d d^* + d^* d=(d+d^*)^2=D^2$ on the exterior algebra bundle 
$\Omega = \oplus \Omega_k$ are finite $v \times v$ matrices, where $\dim(\Omega) = v=\sum_k v_k$, with 
$v_k=\dim(\Omega_k)$ denoting the cardinality of the set $\G_k$ of complete subgraphs $K_k$ of $G$. 
We prove the McKean-Singer formula $\chi(G) = \str(e^{-t L})$ which holds for any 
complex time $t$, where $\chi(G) = \str(1)= \sum_k (-1)^k v_k$ is the Euler characteristic of $G$. 
The super trace of the heat kernel interpolates so the Euler-Poincar\'e formula
for $t=0$ with the Hodge theorem in the real limit $t \to \infty$. 
More generally, for any continuous complex valued function $f$ satisfying $f(0)=0$, 
one has the formula $\chi(G) = \str(e^{f(D)})$.  
This includes for example the Schr\"odinger evolutions 
$\chi(G) = \str(\cos(t D))$ on the graph. After stating some immediate general facts about the
spectrum which includes a perturbation result estimating the Dirac spectral difference of graphs,
we mention as a combinatorial consequence that the spectrum 
encodes the number of closed paths in the simplex space of a graph.
We give a couple of worked out examples and see that McKean-Singer allows to
find explicit pairs of nonisometric graphs which have isospectral Dirac operators. 
\end{abstract} 

\section{Introduction}

Some classical results in differential topology or Riemannian geometry have
analogue statements for finite simple graphs. Examples are 
Gauss-Bonnet \cite{cherngaussbonnet}, Poincar\'e-Hopf \cite{poincarehopf}, 
Riemann-Roch \cite{BakerNorine2007},
Brouwer-Lefschetz \cite{brouwergraph} or Lusternik-Schnirelman \cite{josellisknill}.
While the main ideas of the discrete results are the same as in the continuum, 
there is less complexity in the discrete. \\

We demonstrate here the process of discretizing manifolds using graphs for the 
{\bf McKean-Singer formula} \cite{McKeanSinger}
$$ \chi(G) = \str(e^{-t L})  \; ,  $$
where $L$ is the Laplacian on the differential forms and $\chi(G)$ is the Euler characteristic of 
the graph $G$. 
The formula becomes in graph theory an elementary result about eigenvalues of concrete finite matrices. 
The content of this article is therefore teachable in a linear algebra course.
Chunks of functional analysis like elliptic regularity which are needed in the continuum to 
define the eigenvalues are absent, the existence of solutions to discrete partial differential equations 
is trivial and no worries about smoothness or convergence are necessary.
All we do here is to look at eigenvalues of a matrix $D$ defined by a finite graph $G$. 
The discrete approach not only recovers results about the heat flow, we get results 
about unitary evolutions which for manifolds would require need analytic continuation:
also the super trace of $U(t) = e^{i t L}$ where $L$ is the Laplace-Beltrami operator is the 
Euler characteristic. While $U(t) f$ solves $(i d/dt-L)f=0$, the Dirac wave equation
$(d^2/dt^2+L)=(d/dt - iD)(d/dt + iD) f=0$ is solved by
\begin{equation}
\label{diracsolution}
f(t) = U(t) f(0) = \cos(Dt) f(0) + \sin(D t) D^{-1} f'(0) \; ,
\end{equation}
where $f'(0)$ is on the orthocomplement of the zero eigenspace. Writing $\psi=f-i D^{-1} f'$ 
we have $\psi(t) = e^{i Dt} \psi(0)$. The initial position and velocity of the real flow are 
naturally encoded in the complex wave function. That works also for a
discrete time Schr\"odinger evolution like
$T(f(t),f(t-1)) = (D f(t)-f(t-1), f(t))$ \cite{Kni98} if $D$ is scaled. \\

Despite the fact that the Dirac operator $D$ for a general simple graph 
$G=(V,E)$ is a natural object which encodes much of the geometry of the graph, 
it seems not have been studied in such an elementary setup.  
Operators in \cite{BullaTrenkler,Kenyon,Requardt,Bolte} 
have no relation to the Dirac matrices $D$ studied here. 
The operator $D^2$ on has appeared in a more general setting: 
\cite{Mantuano} builds on work of Dodziuk and Patodi and studies 
the combinatorial Laplacian $(d+d^*)^2$ 
acting on \v{C}ech cochains. This is a bit more general: if the cover of the graph consists 
of unit balls, then the \v{C}ech cohomology by definition agrees with the graph cohomology 
and the \v{C}ech Dirac operator agrees with $D$. Notice also that for the Laplacian $L_0$ 
on scalar functions, the factorization $L_0 = d_0 d_0^*$ is very well known since $d_0$ is the
incidence matrix of the graph. The matrices $d_k$ have been used by Poincar\'e already in 1900.  \\
% H. Poincar\', Second complement a l'analysis situs, Proc. London Math. Soc. 32:277-308 (1900).

One impediment to carry over geometric results from the continuum to the discrete 
appears the lack of a Hodge dual for a general simple graph and the absence of symmetry 
like Poincar\'e duality. Does some geometric conditions have to be imposed on the graph like 
that the unit spheres are sphere-like graphs to bring over results from compact Riemannian 
manifolds to the discrete? The answer is no. The McKean-Singer supersymmetry for eigenvalues holds 
in general for any finite simple graph. 
While it is straightforward to implement the Dirac operator $D$ of a graph in a computer, 
setting up $D$ can be tedious when done by hand because for a complete graph of $n$ nodes, 
$D$ is already a $(2^n-1) \times (2^n-1)$ matrix. 
For small networks with a dozen of nodes already, the Dirac operator acts on a vector space on vector spaces 
having dimensions going in to thousands. Therefore, even before we had the computer routines were in place 
which produce the Dirac operator from a general graph, computer algebra software was necessary to 
experiment with the relatively large matrices first cobbled together by hand. 
Having routines which encode the Dirac operator for a network is useful in many ways,
like for computing the cohomology groups. Such code helped us also to write \cite{josellisknill}.
Homotopy deformations of graphs or nerve graphs defined by a \v{C}ech cover allow to reduce
the complexity of the computation of the cohomology groups. \\

Some elementary ideas of noncommutative geometry can be explained well in this framework.
Let $H=L^2(\G)$ denote the Hilbert space defined by the simplex set $\G$ of the graph and let
$B(H)$ denote the Banach algebra of operators on $H$. It is a Hilbert space itself with the
inner product $(A,B) = tr(A^* B)$. The operator $D \in B(H)$ defines together with a subalgebra
$A$ of $B(H)$, a {\bf Connes distance} on $\G$. Classical geometry is when $A=C(\G)$ is the algebra
of diagonal matrices. In this case, the Connes metric
$\delta(x,y)=\sup_{l \in A, |[D,l]| \leq 1}(l,e_x-e_y)$ between two vertices
is the geodesic shortest distance but $\delta$ extends to a metric on all simplices $\G$.
The advantage of the set-up is that other algebras $A$ define other metrics on $\G$ and more
generally on {\bf states}, elements in the unit sphere $X$ of $A^*$. In the classical case,
when $A=H=C(\G) = R^v$, all states are {\bf pure states}, unit vectors $v$ in $H$ and correspond to states
$l \to tr(l \cdot e_v)=(lv,v)$ in $A^*$. For a noncommutative algebra $A$, elements in $X$ are in general
{\bf mixed states}, to express it in quantum mechanics jargon. We do not make use of this here but
note it as additional motivation to study the operator $D$ in graph theory. \\

\begin{figure}
\scalebox{0.45}{\includegraphics{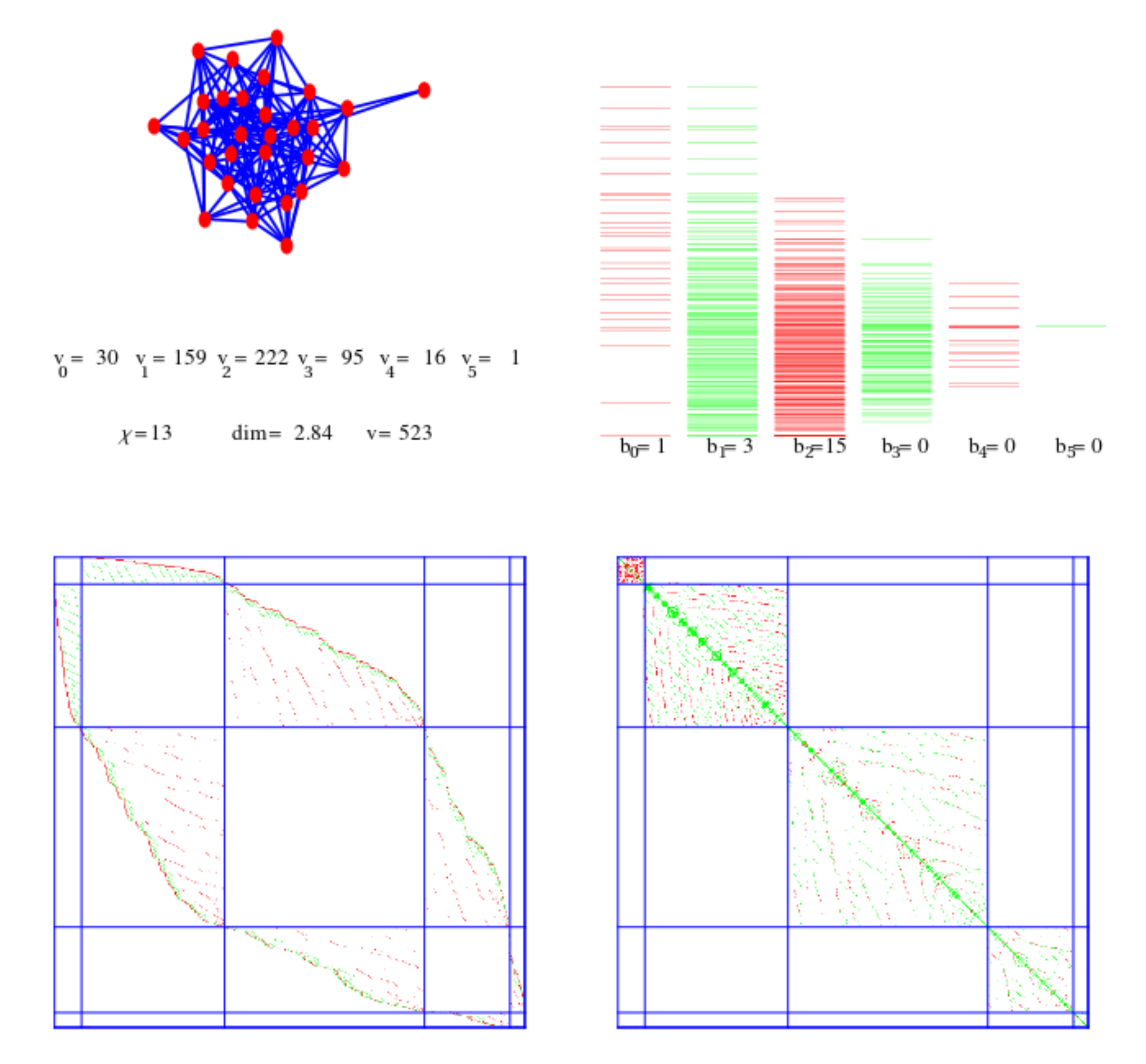}}
\caption{
The lower part shows the Dirac operator $D=d+d^*$ and the Hodge 
Laplace-Beltrami operator $L=d d^* + d^* d$ 
of a random graph with $30$ vertices and $159$ edges. The Dirac operator $D$ of this graph 
is a $523 \times 523$ matrix. 
The Laplacian decomposes into blocks $L_p$, which are the Laplacians on $p$ forms. 
The eigenvalues of $L_k$ shown in the upper right part of the figure are grouped 
for each $p$. The multiplicity of the eigenvalue $0$ is by the Hodge theorem the $p$'th Betti number $b_p$. 
In this case $b_0=1,b_1=3,b_2=13$. Nonzero eigenvalues come in pairs: each fermionic
eigenvalue matches a bosonic eigenvalue. This super symmetry discovered by McKean and Singer shows
that $\str(e^{i t L}) = \sum_{\lambda} (-1)^{deg(\lambda)} e^{i t \lambda} = \chi(G)
=\sum_p (-1)^p b_p = \sum_p (-1)^p v_p$. While all this is here just linear algebra, the McKean formula
mixes spectral theory, cohomology and combinatorics and it applies for any finite simple graph. 
\label{diraccompact}
}
\end{figure}

The spectral theory of Jacobi matrices and Schr\"odinger operators show that the discrete and 
continuous theories are often very close but that it is not always straightforward to translate 
from the continuum to the discrete. The key to get the same results as in the continuum 
is to work with the right definitions. For translating between compact Riemannian manifolds and 
finite simple graphs, the Dirac operator can serve as a link, as we see below. 
While the result in this paper is purely mathematical and just restates a
not a priory obvious symmetry known already in the continuum, 
the story can be interesting for teaching linear algebra or illustrating some 
ideas in physics. On the didactic side, the topic allows to underline one
of the many ideas in \cite{McKeanSinger}, using linear algebra tools only. 
On the physics side, it illustrates classical and relativistic quantum dynamics in a simple setup.
Only a couple of lines in a modern computer algebra system were needed
to realize the Dirac operator of a general graph as a concrete finite matrix and to find
the quantum evolution in equation (\ref{diracsolution}) as well as concrete examples of
arbitrary members of cohomology classes by Hodge theory. The later is also here just a remark
in linear algebra as indicated in an appendix.

\section{Dirac operator}

For a finite simple graph $G=(V,E)$, the exterior bundle 
$\Omega=\oplus_k \Omega_k$ is a finite dimensional vector space
of dimension $n=\sum_{k=0} v_k$, where $v_k=\dim(\Omega_k)$ is the cardinality of $\G_k$ 
the set of $K_{k+1}$ simplices contained in $G$. 
As in the continuum, we have a super commutative multiplication $\wedge$ 
on $\Omega$ but this algebra structure is not used here. 
The vector space $\Omega_k$ consists of all functions on $\G_k$ which are antisymmetric in all $k+1$ 
arguments. The bundle splits into an even $\Omega_b = \oplus_k \Omega_{2k}$ and an odd part 
$\Omega_f=\oplus_k \Omega_{2k+1}$ which are traditionally called the 
bosonic and fermionic subspaces of $\Omega$. The exterior derivative 
$d: \Omega \to \Omega, df(x_0,\dots, x_{n}) = \sum_{k=0}^{n-1} (-1)^k f(x_0,\dots,\hat{x}_k,\dots,x_{n-1})$ 
satisfies $d^2=0$. While a concrete implementation of $d$ requires to give an orientation of each 
element in $\G$, all quantities we are interested in do not depend on this orientation however. It is 
the usual choice of a basis ambiguity as it is custom in linear algebra. 

\begin{defn}
Given a graph $G=(V,E)$, define the {\bf Dirac operator} $D=d+d^*$ and 
{\bf Laplace-Beltrami} operator
$L=D^2 = d d^* + d^* d$. 
The operators $L_p = d_p^* d_p + d_{p-1} d_{p-1}^*$
leaving $\Omega_p$ invariant are called the Laplace-Beltrami operators $L_p$ on $p$-forms. 
\end{defn}

{\bf Examples.} We write $\lambda^{(m)}$ to indicate multiplicity $m$. \\
{\bf 1)} For the complete graph $K_n$ the spectrum of $D$ is $\{-\sqrt{n}^{(2^{n-1}-1)},0, -\sqrt{n}^{(2^{n-1}-1)} \}$. \\
{\bf 2)} For the circle graph $C_n$ the spectrum is $\{\pm \sqrt{2-2 \cos(2\pi k/n)} \},k=0,n-1\;$,
   where the notation understands that $0$ has multiplicity $2$. The product of the nonzero eigenvalues,
   a measure for the complexity of the graph is $n^2$.  \\
{\bf 3)} For the star graph $S_n$ the eigenvalues of $D$ are 
   $\{-\sqrt{n},-1^{(n-1)},0,1^{(n-1)},\sqrt{n} \; \}$. The product of the nonzero eigenvalues is $n$.  \\
{\bf 4)} For the linear graph $L_n$ with $n$ vertices and $n-1$ edges, 
    the eigenvalues of $D$ are the union of $\{0\}$,
    and $\pm \sigma{K}$, where $K$ is the $(n-1) \times (n-1)$ matrix which has $2$ in the
    diagonal and $1$ in the side diagonal. The product of the nonzero eigenvalues is $-n$.  \\

{\bf Remarks.} \\
{\bf 1)} The scalar part $L_0$ is one of
the two most commonly used Laplacians on graphs \cite{Chung97}. The operator $L$ generalizes it
to all $p$-forms, as in the continuum. \\
{\bf 2)} When $d_k$ is implemented as a matrix, it is called a signed
{\bf incidence matrix} of the graph. Since it depends on the choice of orientation for each simplex,
also the Dirac operator $D$ depends on this choice of basis. Both $L$ and 
the eigenvalues of $D$ do not depend on it.  \\
{\bf 3)} The Laplacian for graphs has been introduced by Kirkhoff in 1847 while proving the matrix-tree theorem.
The computation with incidence matrices is as old as algebraic topology.
Poincar\'e used the incidence matrix $d$ already in 1900 to compute Betti numbers \cite{Grone,Stillwell2009}.  \\
{\bf 4)} While $D$ exchanges bosonic and fermionic spaces 
$D: \Omega_b \to \Omega_f, \Omega_f \to \Omega_b$, the Laplacian $L$ 
is the direct sum $L_k: \Omega_k \to \Omega_k$ of Laplacians on $k$ forms. \\
{\bf 5)} The kernel of $L_k$ is called the vector space of {\bf harmonic $k$ forms}. 
Such harmonic forms represent cohomology classes. 
By Hodge theory (see appendix) the dimension of the kernel is equal to the $k$'th Betti number,
the dimension of the $k$'th cohomology group $H^k(G) = \ker(d_k)/\im(d_{k-1})$.  \\
{\bf 6)} Especially, $L_0$ is the graph Laplacian $B-A$, where
$B$ is the diagonal matrix containing the vertex degrees and $A$ is the adjacency matrix of the graph.
The matrix $L_0$ is one of the natural Laplacians on graphs \cite{Chung97}. \\
{\bf 7)} While $d_{p-1} d_{p-1}^*$ always has $p+1$
in the diagonal, the diagonal entries $d_p^* d_p(x,x)$ counts the number of $K_{p+1}$ 
simplices attached to the $p$ simplex $x$
and $d_p d_{p-1}^*(x,x)$ is always equal to $p+1$. The diagonal entries $L_k(x,x)$ 
therefore determine the number of $k+1$ dimensional simplices attached to a 
$k$ dimensional simplex $x$. We will see that the Dirac and Laplacian matrices are
useful for combinatorics when counting closed curves in $\G$ in the same way as adjacency matrices
are useful in counting paths in $G$. \\
{\bf 8)} The operator $D$ can be generalized as in the continuum to
to bundles. The simplest example is to take a $1$-form $A \in \Omega^1$ and to define 
$d_Af = d f + A \wedge f$. This is possible since $\Omega$ is an algebra.  We have then 
the generalized Dirac operator $D_A = d_A + d_A^*$ and $L_A = D_A^2$.
The {\bf curvature operator} $F g = d_A \circ d_A g$ is no more zero in general. 
McKean-Singer could generalizes to this more general setup as $D_A$ provides the super-symmetry.  \\
{\bf 9)} A $1$-form $A$ defines a field $F=dA$ which satisfies the 
{\bf Maxwell equations} $dF=0, d^* F=j$.  Physics asks to find the field $F$, given $j$. 
While this corresponds to the $4$-current in classical electro magnetism which includes charge and
electric currents, it should be noted that no geometric structure is assumed for the graph. 
The Maxwell equations hold for any finite simple graph. It defines the evolution of 
light on the graph. The equation $d^* dA = j$ can in a Coulomb gauge where $d^* A=0$ be written as $L_1 A = j$
where $L_1$ is the Laplacian on $1$ forms. Given a current $j$, we can get $A$ and so the electromagnetic field $F$.
by solving a system of linear equations. This is possible if $j$ is perpendicular to the 
kernel of $L_1$ which by Hodge theory works if $G$ is simply connected.
Linear algebra determines then determines the field $F$, a function on triangles of a graph. 
As on a simply connected compact manifold, a simply connected graph does not feature 
light in vacuum since $LA=0$ has only the solution $A=0$. \\
{\bf 10)} Dirac introduce the Dirac operator $D$ in the continuum to make quantum mechanics
Lorentz invariant. The quantum evolution becomes wave mechanics. Already in the continuum, there is mathematically
almost no difference between the wave evolution with given initial position and velocity of the wave
the Schr\"odinger evolution with a complex wave because both real and imaginary parts of 
$e^{iD t} (u-i D^{-1} v) = \cos(D t) u + \sin(D t) D^{-1} v$ solve the wave equation. The only difference
between Schr\"odinger and wave evolution is that the initial velocity $v$ of the wave 
must be in the orthocomplement of the zero eigenvalue. (This has to be the case also in the continuum,
if we hit a string, the initial velocity has to have average zero momentum in order not to displace the 
center of mass of the string.) A similar equivalence between Schr\"odinger and wave equation holds
for a time-discretized evolution $(u,v) \to (v-Du,u)$.  \\
{\bf 11)} An other way to implement a bundle formalism is to take a unitary group $U(N)$ and replace 
entries $1$ in $d$ with unitary elements $U$ and $-1$ with $-U$. Now $D=d+d^*$ 
and $L=D^2$ are selfadjoint operators on $L^2(\G,C^N)$ for some $N$ and can be implemented
as concrete $v \cdot N$ matrices. The gauge field $U$ defines now curvatures, which is a $U(N)$-valued function
on all triangles of $\G$. Again, the more general operator $D$ is symmetric and supersymmetry
holds for the eigenvalues. Actually, as $\tr(D^n)$ can be expressed as a sum over closed paths, 
the eigenvalues of $D_A$ are just $N$ copies of the eigenvalues of $D$ if the multiplicative curvatures are $1$
(zero curvature) and the graph is contractible. Since ${\tr}(f(D))=0$ for odd functions $f$ and
$\tr(D^2)$ is independent of $U$, the simplest functional on $U$ which involves the curvature is the Wilson 
action ${\rm tr}(D^4)$ which is zero if the curvature is zero. It is natural to ask to minimize this
over the compact space of all fields $A$. More general functionals are interesting like $\det^*(L)$,
the product of the nonzero eigenvalues of $L$; this is an interesting integral quantity in the flat $U=1$
case already: it is a measure for complexity of the graph and combinatorially interesting
since $\det^*(L_0)$ is equal to the number of spanning trees in the graph $G$. \\

\section{McKean-Singer}

Let $G=(V,E)$ be a finite simple graph. If $K_n$ denotes the complete graph with $n$ vertices, 
we write $\G_n$ for the set of $K_{n+1}$ subgraphs of $G$. The set of all simplices 
$\G=\bigcup_n G_n$ is the super graph of $G$ on which the Dirac operator lives. 
The graph $G$ defines $\G$ without addition of more structure but the additional structure
is useful, similarly as tangent bundles are useful for manifolds. 
It is useful to think of elements in $\G$ as elementary units. If $v_n$ is the cardinality of 
$\G_n$ then the finite sum 
$$  \chi(G)=\sum_{n=0}^{\infty} (-1)^n v_n $$ 
is called the Euler characteristic of $G$.   \\

The set $\Omega_n$ of antisymmetric functions $f$ on $\G_n$ is a linear space of $k$-forms. 
The exterior derivative $d: \Omega_n \to \Omega_{n+1}$ defined by 
$$ df(x_0,x_1,\dots, x_{n+1}) = \sum_{k} (-1)^k f(x_0,\dots,\hat{x}_k,\dots,x_{n+1}) $$ 
is a linear map which satisfies $d \circ d=0$.
It defines the cohomology groups $H^n(G) = {\rm ker}(d_{n})/{\rm im}(d_{n-1})$ of dimension $b_n(G)$
for which the Euler-Poincar\'e formula $\chi(G) = \sum_n (-1)^n b_n(G)$ holds. 
The direct sum $\Omega = \oplus_n \Omega_n$ is the discrete analogue of the exterior bundle.  \\

McKean and Singer have noticed in the Riemannian setup that the heat flow $e^{-t L}$
solving $(d/dt+L)f=0$ has constant super trace $\str(L(t))$. The reason is that the nonzero
bosonic eigenvalues can be paired bijectively with nonzero fermionic eigenvalues.
It can be rephrased that $D$ is an isomorphism from $\Omega_b^+ \to \Omega_f^+$,
where $\Omega_b^+$ is the orthogonal complement of the zero eigenspace in $\Omega_b$ and
$\Omega_f^+$ the orthogonal complement of the zero eigenspace in $\Omega_f$.
In the limit $t \to \infty$ we get the dimensions of the harmonic forms; in the
limit $t \to 0$ we obtain the simplicial Euler characteristic.

\begin{thm}[McKean-Singer]
For any complex valued continuous function $f$ on the real line satisfying $f(0)=0$ we have
$$  \chi(G) = \str( \exp(f(D)) ) \; . $$ 
Especially, $\str(\exp(-t L))= \chi(G)$ for any complex $t$. 
\end{thm}

\begin{proof}
Since $D=d+d^*$ is a symmetric $v \times v$ matrix, the kernel of $D$ and the kernel of $L=D^2$ 
are the same. The functional calculus defines $f(D)$ for any complex valued continuous function.
Since $D$ is normal, $D D^* = D^* D = -L$, it 
can be diagonalized $D=U^* E D$ with a diagonal matrix $E$ and defines $f(D)=U^* f(E) U$ for any 
continuous function. 
Because $f$ can by Weierstrass be approximated by polynomials and the diagonal entries of $D^{2k+1}$ 
are empty, we can assume that $f(D)=g(D^2)=g(L)$ for an even function $g$. 
Let $\Omega^+$ be the subspace of $\Omega$ spanned by nonzero eigenvalues of $L$. 
Let $\Omega^+_p$ be the subspace generated by eigenvectors to nonzero eigenvalues of $L$ on $\Omega_p$
and $\Omega^+_f$ be the subspace generated by eigenvectors to nonzero eigenvalues of $L$ on $\Omega_f$. 
Since $D$ commutes with $L$, each eigenvector $f$ of $g(L)$ on $\Omega_f$ has an eigenvector $D f$ of $g(L)$ 
on $\Omega_p$. Since $D: \Omega^+_p \to \Omega^+_m$ is invertible, there is a 
bijection between fermionic and bosonic eigenvalues. 
Each nonzero eigenvalue appears the same number of times on the fermionic and bosonic part. 
\end{proof}

{\bf Remarks.} \\
{\bf 1} By definition, $\str(1) = \sum_{k=0} (-1)^k v_k$ agrees with the Euler characteristic.
It can be seen as an {\bf analytic index} $\ind(D)$ of the restricted Dirac operator $\Omega_b \to \Omega_f$
because $\ker(D)$ is the space of harmonic states in $\Omega_b$ and $\coker(D) = \ker(D^*)$ is
the dimension of the fermionic harmonic space.\\
{\bf 2)} The original McKean-Singer proof works in the graph theoretical setup too 
as shown in the Appendix.  \\
{\bf 3)} In \cite{Kni98} we have for numerical purposes discretized the Schroedinger flow. This shows that
we can replace the flow $e^{tD}$ by a map $T(f,g) = (g-D f,f)$ with a suitably rescaled $D$
which is dynamically similar to the unitary evolution and has the property that the system has
{\bf finite propagation speed}. The operator $T^2(f,g)=(f-D(g-Df),g-Df)=(f+Lf,g)-(Dg,Df)$ 
has the super trace $\str(T^2) = \str((1,1)) = 2 \chi(G)$ so that also this 
discrete time evolution satisfies the McKean-Singer formula. \\
{\bf 4)} The heat flow proof interpolating between the identity and the projection onto
harmonic forms makes the connection between Euler-Poincar\'e and Hodge more natural.
While for $t=0$ and $t=\infty$ we have a Lefshetz fixed point theorem,
for $0<t<\infty$ it can be seen as an application of the Atiyah-Bott generalization of
that fixed point theorem. In the discrete, Atiyah-Bott is very similar
to Brouwer-Lefshetz \cite{brouwergraph}. For $t=0$ the Lefshetz fixed point theorem expresses the
Lefshetz number of the identity as the sum of the fixed points by Poincar\'e-Hopf.
In the limit $t \to \infty$, the Lefshetz fixed point theorem
sees the Lefshetz number as the signed sum over all fixed points which are harmonic forms.
For positive finite $t$ we can rephrase McKean-Singer's result that the Lefshetz number 
of the Dirac bundle automorphism $e^{-t L}$ is time independent. 

\section{The spectrum of $D$}

We start with a few elementary facts about the operator $D$: 

\begin{propo}
Let $\vec{\lambda}$ denote the eigenvalues of $D$ and let $\deg$ denote the maximal degree of $G$.
If $\lambda$ is an eigenvalue, then $-\lambda$ is an eigenvalue too so that
${\rm E}[\lambda]=\sum_i \lambda_i=0$.
\end{propo}

\begin{proof} 
If we split the Hilbert space as $\Omega=\Omega_p \oplus \Omega_f$ 
then  $D=\left[ \begin{array}{cc} 0 & A^* \\ A & 0 \end{array} \right]$, 
where $A=d+d^*$ is the annihilation operator and $A^*$ which maps bosonic to fermionic states and 
$A$ is the creation operator. Define $P=\left[ \begin{array}{cc} 1 & 0 \\ 0 & -1 \\ \end{array} \right]$. 
Then $L=D^2, P^2=1, DP+PD=0$ is called supersymmetry in 0 space dimensions (see \cite{Witten1982,Cycon}). 
If $Df=\lambda f$, then $PD Pf = -D f = -\lambda f$. 
Apply $P$ again to this identity to get $D (Pf) = - \lambda (Pf)$. This shows that
$Pf$ is an other eigenvector. 
\end{proof}

\begin{propo}
Every eigenvalue $\lambda$ is  contained in $[-\sqrt{2\deg},\sqrt{2\deg}]$ so that
every eigenvalue of $L=D^2$ is contained in the interval $[0,2\deg]$
and ${\rm Var}[\lambda] = \sum_{i=1}^n \lambda_i^2/n \leq 2\deg$.
\end{propo}

\begin{proof}
Because all entries of $D$ are either $-1$ and $1$ and because there are $d$ such entries, each row 
or column has maximal length $\sqrt{d}$.  
That the standard deviation is $\leq \sqrt{2\deg}$ follows since the "random variables"
$\lambda_j$ take values in $[-\sqrt{2\deg},\sqrt{2\deg}]$. Also a theorem of Schur
\cite{SimonTrace} confirms that $\sum_i \lambda_i^2 \leq \sum_{i,j} |D_{ij}|^2 \leq n \cdot 2\deg$.
\end{proof}

{\bf Remarks}. \\
{\bf 1)} It follows that the spectrum of $D$ is determined by the spectrum of $L$. If $\mu_j \in [0,\deg]$ 
are the eigenvalues of $L$, then $\pm \sqrt{m_j}$ are the eigenvalues of $D$. It follows already from the
reflection symmetry of the eigenvalues of $D$ that the positive eigenvalues of $L$ appear {\bf in pairs}. 
The McKean-Singer statement is stronger than
that. It tells that if one member of the pair is in the bosonic sector, the other is in the fermionic one. \\
{\bf 2)} The Schur argument given in the proof is not quite irrelevant when looking at spectral statistics. 
Since the average degree in the Erdoes-R\'enyi probability space $E(n,1/2)$ is of the order $\log(n)$
the standard deviation of the eigenvalues is of the order $2\log(n)$ even so we can have eigenvalues
arbitrarily close to $\sqrt{2(n-1)}$.  \\
{\bf 3)} Numerical computations of the spectrum of $D$ for large random matrices is difficult because
the Dirac matrices are much larger than the adjacency matrices of the graph. It would be interesting
however to know more about the distribution of the eigenvalues of $D$ for large matrices. \\
{\bf 4} Sometimes, an other Laplacian $K_0$ defined as follows for graphs. 
Let $V_0(x)$ denote the degree of a vertex $x$. Define operator $K_{xx}=1$ if $V_0(x)>0$ and
$K_{xy} = -(V_0(x) V_0(y))^{-1/2}$ if $(x,y)$ is an edge. This operator satisfies $K=C L_0 C$, where
$C$ is the diagonal operator for which the only nonzero entries are $V_0(x)^{-1/2}$ if $V_0(x)>0$.
While the spectrum of $L_0$ is in $[0,2\deg]$ the spectrum of $K$ is in $[0,2]$. Since $L_0$ has integer
entries, it is better suited for combinatorics.
In the following examples of spectra we use the notation $\lambda^{(n)}$ indicating that the eigenvalue 
$\lambda$ appears with multiplicity $n$. \\
For the complete graph $K_{n+1}$ the spectrum of $K_0$ is $\{ 0,(n+1)/n)^{(n)} \}$ 
while the spectrum of $L_0$ is $\{ 0,n^{(n)} \; \}$.
For a cycle graph $C_n$ the spectrum of $K_0$ is $\{ 1-\cos(2\pi k/n) \; \}$ while the spectrum of
$L_0$ is $2-2 \cos(2\pi k/n)$. This square graph $C_4$  shows that the estimate $\sigma(L) \subset [0,2\deg]$ is
optimal. For a star graph $S_n$ which is an example of a not a vertex regular graph,
the eigenvalues of $K_0$ are $\{ 0,1^{(n-2)},2 \; \}$ while the spectrum of
$L_0$ is $\{ 0,1^{(n-2)},n \; \}$. \\

\begin{propo}
The number of zero eigenvalues of $D$ is equal to the sum of all the Betti numbers
$\sum_k b_k$. 
\end{propo}
\begin{proof}
This follows from the Hodge theorem (see Appendix): the dimension of the kernel of $L_k$
is equal to $b_k$. 
\end{proof} 

\begin{defn}
The {\bf Dirac complexity} of a finite graph is defined as the product of the 
nonzero eigenvalues of its Dirac operator $D$. 
\end{defn}

The Euler-Poincar\'e identity assures that $\sum_{i=0}^{\infty} (-1)^i (v_i - b_i)=0$.
If we ignore the signs, we get a number of interest: 

\begin{lemma}
The number of nonzero eigenvalue pairs in $D$ is the
{\bf sign-less Euler-Poincar\'e number} $\sum_{i=0}^{\infty} (v_i - b_i)/2$.
\end{lemma}
\begin{proof}
The sum $\sum_i v_i = v$ is the total number of eigenvalues and 
by the previous proposition, the sum $\sum_i b_i$ is the total number
of zero eigenvalues. Since the eigenvalues of $D$ come in pairs $\pm \lambda$,
the number of pairs is the sign-less Euler-Poincar\'e number. 
\end{proof} 

\begin{coro}
The sign-less Euler-Poincar\'e number is even if and only if the 
Dirac complexity is positive. 
\end{coro} 
\begin{proof}
Arrange the product of nonzero eigenvalues of $D$ as a product of pairs $-\lambda_j^2$. 
\end{proof}

\begin{coro}
If $G$ is a triangularization of a sphere and which
has an even number of edges, then the Dirac complexity is positive. 
\end{coro}
\begin{proof}
We have $\chi(G) = v_0-v_1+v_2=2$ and since $b_0=b_2=1$ and $b_1=0$
we have $\sum_i b_i=2$. We can express now the sign-less Euler-Poincar\'e number
in terms of the number of edges: 
$$ [(v_0 +v_1 + v_2) - (b_0 + b_1 + b_2)]/2 = [(2+2 v_1) - 2]/2 = 2 v_1/2 = v_1 \; . $$
\end{proof}

The following explains why the dodecahedron or cube have negative complexity: 

\begin{coro}
If $G$ is a connected graph without triangles which has an even number
of vertices, then the Dirac complexity is negative. 
\end{coro}
\begin{proof}
$\chi(G) = v_0-v_1 = b_0-b_1 = 1-b_1$. The sign-less Euler-Poincar\'e number is
$[(v_0+v_1)-(b_0+b_1)]/2 = [1-b_1+2v_1 - (1+b_1)]/2 = v_1-b_1 = v_0-b_0 = v_0-1$.  
\end{proof}

For cyclic graphs $C_n$ the complexity is $n^2$ if $n$ is odd and $-n^2$ if
$n$ is even.  For star graphs $S_n$, the complexity is $n$ is odd and $-n$
if $n$ is even. 

\begin{coro}
For a tree, the Dirac complexity is positive if and only if the number of edges
is even.
\end{coro}
\begin{proof}
The Dirac complexity is still $v_1-b_1$ as in the previous proof but now $b_1=0$ so 
that it is $v_1$. 
\end{proof}

We have computed the complexity for all Platonic, Archimedian and Catalan solids
in the example section. All these 31 graphs have an even number $v_1$ of edges and an
even number $v_1$ of vertices. 

\section{Perturbation of graphs}

Next we estimate the distance between the spectra of different graphs:
We have the following variant of Lidskii's theorem which I learned from \cite{Last1995}:

\begin{lemma}[Lidskii]
For any two selfadjoint complex $n \times n$ matrices $A$ and $B$ with
eigenvalues $\alpha_1 \leq \alpha_2 \leq \dots \leq \alpha_n$ and
$\beta_1 \leq \beta_2 \leq \dots \leq \beta_n$, one has
$$ \sum_{j=1}^n |\alpha_j - \beta_j| \leq \sum_{i,j=1}^{n} |A-B|_{ij} \; . $$
\end{lemma}
\begin{proof}
Denote with $\gamma_i \in R$ the eigenvalues of the selfadjoint
matrix $C:=A-B$ and let $U$ be the unitary matrix diagonalizing $C$
so that ${\rm Diag}(\gamma_1, \dots ,\gamma_n)=UCU^*$. We calculate
\begin{eqnarray*}
 \sum_i |\gamma_i| &=& \sum_i (-1)^{m_i} \gamma_i
                    = \sum_{i,k,l} (-1)^{m_i} U_{ik} C_{kl} U_{il} \\
                    &\leq& \sum_{k,l} |C_{kl}| \cdot
                                    |\sum_i (-1)^{m_i} U_{ik} U_{il}|
                    \leq \sum_{k,l} |C_{kl}| \; .
\end{eqnarray*}
The claim follows now from Lidskii's inequality
$\sum_j |\alpha_j-\beta_j| \leq \sum_j |\gamma_j|$
(see \cite{SimonTrace})
%Theorem 1.20
\end{proof}

This allows to compare the spectra of Laplacians $L_0$ of graphs which are close. Lets define the following
metric on the Erdoes-R\'enyi space $G(n)$ of graphs of order $n$ on the same vertex set. 
Denote by $d_0(G,H)$ the number of edges at which $G$ and $H$ differ divided by $n$. 
Define also a metric between their adjacency 
spectra $\vec{\lambda}, \vec{\mu}$ as 
$$  d_0(\vec{\lambda},\vec{\mu}) = \frac{1}{n} \sum_{j=1}^n |\lambda_j-\mu_j| \; . $$

\begin{coro}
If the maximal degree in either $G,H$ is $\deg$, then the adjacency spectra distance
can be estimated by 
$$ d_0(\vec{\lambda},\vec{\mu})  = 2 \deg d_0(\sigma_0(G),\sigma_0(H)) \; . $$
\end{coro}
\begin{proof}
We only need to verify the result for $d_0(G,H)=1$, since the left hand side
is the $l^1$ distance between the vectors $\lambda,\mu$ and the triangle inequality
gives then the result for all $d_0(G,H)$. For $d_0(G,H)=1$, the matrix $C=B-A$ 
differs by $1$ only in $2 \deg$ entries. The Lidskii lemma implies the result.
\end{proof}

{\bf Remarks.} \\
{\bf 1)} We take a normalized distance between graphs because this is better suited for taking graph limits.\\
{\bf 2)} As far as we know, Lidskii's theorem has not been used yet in the spectral theory of graphs. 
We feel that it could have more potential, especially when looking at random graph settings 
\cite{randomgraph}.   \\
3) For the Laplace operator $L_0$ this estimate would be more complicated since the 
diagonal entries can differ by more than $1$. It is more natural therefore 
to look at the Dirac matrices for which the entries are only $0,1$ or $-1$. \\

To carry this to Dirac matrices, there are two things to consider. First, the 
matrices depend on a choice of orientation of the simplices which requires to chose the 
same orientation if both graphs contain the same simplex. Second, the Dirac matrices 
have different size because $D$ is a $v \times v$ matrix if $v$ is the total number of 
simplices in $G$. Also this is no impediment: take the union of all simplices which occur
for $G$ and $H$ and let $v$ denote its cardinality. We can now write down possibly
{\bf augmented} $v \times v$ matrices $D(G),D(H)$
which have the same nonzero eigenvalues than the original Dirac operator.
Indeed, the later matrices are obtained from the augmented matrices by deleting the
zero rows and columns corresponding to simplices which are not present. 

\begin{defn}
Define the {\bf spectral distance} between $G$ and $H$ as
$$  \frac{1}{v} \sum_{j=1}^v |\lambda_j-\mu_j|   \; ,  $$
where $\lambda_j, \mu_j$ are the eigenvalues of the augmented Dirac matrices 
$G,H$, which are now both $v \times v$ matrices. 
\end{defn}

\begin{defn}
Define the {\bf simplex distance} $d(G,H)$ of two graphs $G,H$ with vertex set $V$ as $(1/v)$ times
the number of simplices of $G,H$ which are different in the complete graph on $V$. Here $v$ is the
total number of simplices in the union when both graphs are considered subgraphs of the complete
graph on $V$. 
\end{defn}

\begin{defn}
The {\bf maximal simplex degree} of a simplex $x$ of dimension $k$ in a graph $G$ is the 
sum of the number of simplices of dimension $k+1$ which contain $x$ and the 
sum of the number of simplices of dimension $k-1$ which are included in $x$. 
\end{defn}

In other words, the maximal simplex degree is the number of nonzero entries $d_{x,y}$ in a column $x$
of the incidence matrix $d$. We get the following perturbation result:  

\begin{coro}
The Dirac spectra of two graphs $G,H$ with vertex set $V$ satisfies
$$ d(\sigma(G),\sigma(H)) \leq 2 \deg \cdot d(G,H) \; , $$
if $\deg$ is the maximal simplex degree. 
\end{coro}
\begin{proof}
Since we have chosen the same orientation for simplices which are present in both graphs, 
this assures that the matrix $D(G)-D(H)$ has entries of absolute value $\leq 1$ and are nonzero
only at $D_{x,y}$, where an incidence happens exactly at one of the graphs $G,H$. 
The value $2 \deg d(G,H)$ is an upper bound for $\frac{1}{v} \sum_{k,l} |D(G)_{kl}-D(H)_{kl}|$,
which by Lidskii is an upper bound for the spectral distance 
$\frac{1}{v} \sum_{j=1}^v |\lambda_j-\mu_j|$.
\end{proof}

{\bf Remarks.}  \\
{\bf 1)} While the estimate is rough in general, it can become useful when looking at graph limits
or estimating spectral distances between Laplace eigenvalues of different regions.
If we look at graphs of fixed dimensions like classes of triangularizations of a manifold $M$, then 
$\deg$ is related to the dimension of $M$ only and both sides of the inequality 
have a chance to behave nicely in the continuum limit. \\
{\bf 2)} For large graphs which agree in many places,
most eigenvalues of the Laplacian must be close. The result is convenient to 
estimate the spectral distance between the complete graph $K_n$ and $G$. In that case,
$2 \cdot \deg = 2 (2^{n-1}-1) \leq 2^n$ and we have 
$$ d(\sigma(G),\sigma(K_n)) \leq d(G,K_n) \; . $$
Since the Dirac spectrum of $K_n$ is contained is the set 
$\{-\sqrt{n},0,\sqrt{n} \; \}$.
It follows that Erdoes-Renyi graphs in $G(n,p)$ for probabilities $p$ close to $1$ have a Dirac 
spectrum concentrated near $\pm \sqrt{n}$. 
If $G$ has $m$ simplices less than $K_n$, then $d(G,K_n) \leq m/2^n$.  \\

{\bf Examples.}  \\
{\bf 1)} Let $G$ be the triangle and $H$ be the line graph with three vertices. Then $v=7$.
We have $\deg=3$.  Since the graphs differ on the triangle and one edge only, we have
$d(G,H)=2/7$ and $2 \deg \cdot d(G,H)=12/7=1.71 \dots $. The augmented Dirac matrices are
$$   D_G = \left[
                  \begin{array}{ccccccc}
                   0 &  0 & 0 & -1 & -1 &  0 &  0 \\
                   0 &  0 & 0 &  0 &  1 & -1 &  0 \\
                   0 &  0 & 0 &  1 &  0 &  1 &  0 \\
                  -1 &  0 & 1 &  0 &  0 &  0 &  1 \\
                  -1 &  1 & 0 &  0 &  0 &  0 & -1 \\
                   0 & -1 & 1 &  0 &  0 &  0 & -1 \\
                   0 &  0 & 0 &  1 & -1 & -1 &  0 \\
                  \end{array}
                  \right], 
     D_H = \left[
                  \begin{array}{ccccccc}
                   0 &  0 & 0 & -1 &  0 &  0 &  0 \\
                   0 &  0 & 0 &  0 &  0 & -1 &  0 \\
                   0 &  0 & 0 &  1 &  0 &  1 &  0 \\
                  -1 &  0 & 1 &  0 &  0 &  0 &  0 \\
                   0 &  0 & 0 &  0 &  0 &  0 &  0 \\
                   0 & -1 & 1 &  0 &  0 &  0 &  0 \\
                   0 &  0 & 0 &  0 &  0 &  0 &  0 \\
                  \end{array}
                  \right] \; .  $$
The spectrum of the Dirac matrix of $G$ is
$\lambda = \{ -\sqrt{3},-\sqrt{3},-\sqrt{3},0,\sqrt{3},\sqrt{3},\sqrt{3} \; \}$
The spectrum of the augmented Dirac matrix of $H$ is
$\mu=\{ -\sqrt{3},-1,0,0,0,1, \sqrt{3} \; \}$.
The spectral distance is $d(\sigma(G),\sigma(H) ) = (2+2 \sqrt{3})/7=0.781 \dots$.  \\
{\bf 2)} If $G$ is $K_n$ and $H$ is $K_{n-1}$ then $v=2^n-1$.
We have $d(G,H) = [(2^{n}-1)-(2^{n-1}-1)]/(2^n-1) = 2^{n-1}/(2^n-1) \sim 1/2$
and $2 \deg = (2 (2^{n-1}-1)) \sim 2^n$ and so $2 \deg d(G,H) \sim 2^(n-1)$. 
The Dirac eigenvalues differ by $|\sqrt{n}-\sqrt{n-1})|$ at $(2^{n-1}-2)$ places
and by $\sqrt{n}$ at $(2^n-1)-(2^{n-1}-1)=2^{n-1}$ places. The spectral distance 
is $|\sqrt{n}-\sqrt{n-1})| \frac{(2^{n-1}-2) }{2^n-1} + \sqrt{n} \frac{2^{n-1}}{2^n-1}$
which is about $\sqrt{n}/2$.  Taking away one vertex together with all the edges is 
quite a drastic perturbation step. It gets rid of a lot of simplices and changes 
the dimension of the graph.  \\
{\bf 3)} Let $G$ be the wheel graph $W_4$ with 4 spikes. It is the simplest model for a 
planar region with boundary. Let $H$ be the graph where we make a pyramid extension over one
of the boundary edges. This is a homotopy deformation. The graph $G$ has $17$ simplices and the 
graph $H$ has $v=21$ simplices. The graph $H$ has one vertex, two edges and one triangle more 
than $G$ so that $d(G,H) = 4/21 = 0.190 \dots$. The maximal degree is $\deg=4$ so that 
the right hand side of the estimate is $32/21=1.52 \dots$. The eigenvalues of the Dirac operator
$D(H)=$ \\
\begin{center}
\begin{tiny}
$\left[
\begin{array}{ccccccccccccccccccccc}
0&0&0&0&0&0&a&a&a&a&0&0&0&0&0&0&0&0&0&0&0\\
0&0&0&0&0&0&0&1&0&0&a&a&a&0&0&0&0&0&0&0&0\\
0&0&0&0&0&0&0&0&0&0&0&1&0&a&a&0&0&0&0&0&0\\
0&0&0&0&0&0&1&0&0&0&0&0&0&1&0&a&0&0&0&0&0\\
0&0&0&0&0&0&0&0&1&0&1&0&0&0&0&0&0&0&0&0&0\\ 
0&0&0&0&0&0&0&0&0&1&0&0&1&0&1&1&0&0&0&0&0\\ 
a&0&0&1&0&0&0&0&0&0&0&0&0&0&0&0&0&0&a&0&0\\
a&1&0&0&0&0&0&0&0&0&0&0&0&0&0&0&a&a&0&0&0\\
a&0&0&0&1&0&0&0&0&0&0&0&0&0&0&0&1&0&0&0&0\\
a&0&0&0&0&1&0&0&0&0&0&0&0&0&0&0&0&1&1&0&0\\
0&a&0&0&1&0&0&0&0&0&0&0&0&0&0&0&a&0&0&0&0\\
0&a&1&0&0&0&0&0&0&0&0&0&0&0&0&0&0&0&0&a&0\\
0&a&0&0&0&1&0&0&0&0&0&0&0&0&0&0&0&a&0&1&0\\
0&0&a&1&0&0&0&0&0&0&0&0&0&0&0&0&0&0&0&0&a\\
0&0&a&0&0&1&0&0&0&0&0&0&0&0&0&0&0&0&0&a&1\\
0&0&0&a&0&1&0&0&0&0&0&0&0&0&0&0&0&0&a&0&a\\
0&0&0&0&0&0&0&a&1&0&a&0&0&0&0&0&0&0&0&0&0\\
0&0&0&0&0&0&0&a&0&1&0&0&a&0&0&0&0&0&0&0&0\\
0&0&0&0&0&0&a&0&0&1&0&0&0&0&0&a&0&0&0&0&0\\
0&0&0&0&0&0&0&0&0&0&0&a&1&0&a&0&0&0&0&0&0\\
0&0&0&0&0&0&0&0&0&0&0&0&0&a&1&a&0&0&0&0&0\\
\end{array}
   \right]$ \\
\end{tiny}
\end{center}
(where, we wrote $a=-1$ for typographic reasons) of $H$ are 
\begin{small}
$\{-2.370,-2.302,-2.266,-2.$ $,-1.913,-1.839,$ $-1.732,-1.5,-1.302,$
    $-0.9296,0,0.9296,1.302,$ $1.529,1.732,1.839,1.913,$ $2,2.266,2.302,2.370\}$. 
\end{small}
The eigenvalues of the augmented Dirac operator of $G$ are 
\begin{small}
$\{-2.236,-2.236,-2.236$ $,-1.732,$ $-1.732,-1.732,$ $-1.732,-1,0,0,$
   $0,0,0,1$ $,1.732,1.732,$ $1.732,1.732,2.236,2.236,2.236\}$.
\end{small}
The spectral differences is $0.337998 \dots$. We were generous with the
degree estimate. The new added point adds simplices with maximal degree $3$ so that the proof
of the perturbation result could be improved to get an upper bound $8/7 = 1.142 \dots$. 
It is still by a factor 3 larger than the spectral difference, but we get the idea:
if $H$ is a triangularization of a large domain and we just move
the boundary a bit by adding a triangle then the spectrum almost does not budge. 
This will allow us to study spectral differences 
$\lim_{n \to \infty} \frac{1}{\lambda_n} \sum_{j=1}^n |\lambda_j-\mu_j|$
of planar regions in terms of the area of the symmetric difference. 

\section{Combinatorics}

Adjacency matrices have always served as an algebraic bridge to study graphs. 
The entry $A^n_{ij}$ has the interpretation as the number of paths in $G$
starting at a vertex $i$ and ending at a vertex $j$. From the adjacency matrix $A$, 
the Laplacian $L_0=B-A$ is defined. Similarly as for $L_0$, 
we can read off geometric quantities also from the diagonal entries of the full
Laplacians $L=D^2$ on $p$ forms.

\begin{defn}
The number of $(p+1)$-dimensional simplices which contain the $p$-dimensional 
simplex $x$ is called the {\bf p-degree} of $x$. It is denoted $\deg_p(x)$. 
\end{defn}

{\bf Examples.} \\
{\bf 1)} For $p=0$, the degree $\deg_0(x)$ is the usual degree vertex degree $\deg(x)$ 
of the vertex $x$. \\
{\bf 2)} For $p=1$ the degree $\deg_1(x)$ 
is the number of triangles which are attached to an edge $x$. 

\begin{propo}[Degree formulas]
For $p>0$ we have ${\rm deg}_p(x) = L_p(x,x)-(p+1)$. 
For $p=0$ we have ${\rm deg}_0(x) = L_0(x,x)$.
\end{propo}
\begin{proof}
Because $d_p$ has $p+1$ nonzero entries $1$ or $-1$ in each row $x$, we have $d_p^* d_p (x,x) = p+1$. 
The reason for the special situation $p=0$ is that $L_0$ is the only Laplacian which 
does not contain a second part $d_0 d_0^*$ because $d_0^*$ is zero on scalars: 
$$ L_0 = d_0^* d_0, \;  L_1= d_1^* d_1 + d_0 d_0^*,  \; L_2=d_2^* d_2 + d_2 d_2^*  \;  etc.  \; . $$
\end{proof}

{\bf Remarks.} \\
{\bf 1)} The case $p=0$ is special because $L_0$ only consists of $d_0^* d_0$ while 
$L_p$ with $p>0$ has two parts $d_{p-1} d_{p-1}^* + d_{p}^* d_{p}$. The degree formulas
actually count closed paths of length $2$ starting at $x$. For $p>0$, paths can 
also lower the dimension. For example, for $p=1$, where we start with an edge, then there are two paths
which start at the edge, chose a vertex and then get back to the edge. This explains the 
correction $p+1$. \\
{\bf 2)} For $p=1$, we can count the number of triangles attached to an edge $x$ 
with ${\rm deg}_1(x) = L_1(x,x)+2$. \\
{\bf 3)} A closed path interpretation of the diagonal entries will generalize the statement. 
Paths of length $2$ count adjacent simplices in $\G$. 

We can read off the total number $v_p$ of $p$-simplices in $G$ from the trace of $L_p$. This 
generalizes the Euler handshaking result that the sum of the degrees of a finite simple graph
is twice the number of edges: 

\begin{coro}[Handshaking]
$\tr(L_p) = (p+2) v_{p+1}$.
\end{coro}
\begin{proof} 
Summing up ${\rm deg}_p(x) = L_p(x,x)-(p+1)$ gives 
$v_{p+1} = {\rm tr}(L_p) - (p+1) v_{p+1}$. 
\end{proof} 

We write just $1$ for the identity matrix. The formula
$$ \str(L+1)=\chi(G) $$ 
follows from McKean-Singer because $\str(1)=\chi(G), \str(L)=0$. 
Combinatorically it is equivalent to a statement which 
we can verify directly and manifests in cancellations of traces: 
\begin{eqnarray*}
\tr(L_0+1) &=& 2 v_1 + v_0 \\
\tr(L_1+1) &=& 3 v_2 + 3 v_1 \\
\tr(L_2+1) &=& 4 v_3 + 4 v_2 \\
           &\dots&  \; .
\end{eqnarray*}
Adding this up gives 
$\str(L+1) = \tr(L_0+1)-\tr(L_1+1) + \dots = v_0-v_1+v_2-v_3 + \dots = \chi(G)$.  \\

An other consequence is
\begin{eqnarray*}
\tr(L_0)/2   &=& v_1  \\
\tr(L_1-2)/3 &=& v_2 \\
\tr(L_2-3)/4 &=& v_3  \\
             &\dots&  \; . 
\end{eqnarray*}

{\bf Remarks.} \\
{\bf 1)} Each identity $\str(L^k+1) = \chi(G)$ produces 
some "curvatures" on the super graph $\G$ satisfying 
$$ \chi(G) = \sum_{x \in \G} \kappa(x) $$
the case $k=0$ being trivial giving $\kappa(x) = (-1)^{dim(x)}$ for which 
Gauss-Bonnet is the definition of the Euler characteristic and where 
$k=1$ is the case just discussed. \\

While the adjacency matrix $A$ of a graph has $\tr(A^k)$ as the
number of closed paths in $G$ of length $k$, the interpretation of 
$\tr(L_0^k) = \tr((B-A)^k)$ becomes only obvious when looking at it in more generally 
when looking the full Laplacian $L$ as we do here. Instead of finding an interpretation
where we add loops to the vertices, it is more natural to look at paths in $\G$: 

\begin{defn}
A {\bf path} in $\G = \bigcup \G_k$ as a sequence of simplices
$x_k \in \G$ such that either $x_k$ is either strictly contained in $x_{k+1}$ or that 
$x_{k+1}$ is strictly contained in $x_k$ and a path starting in $\G_k$ can additionally to 
$\G_k$ only visit one of the neighboring spaces $\G_{k+1}$ or $\G_{k-1}$ 
along the entire trajectory. 
\end{defn}

The fact that this random walk on $\G$ can not visit three different dimension-sectors 
$$  \G_{k-1},\G_{k},\G_{k+1}  $$
is a consequence of the identities $d^2=(d^*)^2=0$; a path visiting three different sectors
would have cases, where $d$ or $d^*$ appears in a pair in the expansion of $(d+d^*)^{2k}$.
Algebraically it manifests itself in the fact that $L^n$ is always is a block matrix 
for which each block $L_k^n$ leaves the subspace $\Omega_k$ of $k$-forms invariant. \\

{\bf Examples.} \\
{\bf 1)} For a graph without triangles, a path starting at a vertex $v_0$ is a sequence
$v_0,e_1,v_1,e_2,v_2 \dots $. Every path in $\G$ of length $2n$ corresponds to a path
of length $n$ in $G$.  \\
{\bf 2)} For a triangular graph $G$, there are three closed paths of length $2$ starting
an edge $e=(v_1,v_2)$. The first path is $e,v_1,e$, the second $e,v_2,e$ and the third
is $e,t,e$ where $t$ is the triangle. \\
{\bf 3)} Again for the triangle, there are $6$ closed paths of length $4$ starting at a
vertex: four paths crossing two edges $v,e_i,v,e_j,v$ and two paths crossing the same edge twice 
$v,e_i,v_1,e_i,v$. There are $9$ paths of length $4$ starting at an edge. There are 
four paths of the form $e,v_j,e,v_k,e$ and four paths of the form $e,v_i,e_j,v_i,e$
and one path $e,t,e,t,e$. 

\begin{propo}[Random walk in $\G$]
The integer $D^k_{xy}$ is the number of paths of length $k$ in $\G$ starting at a simplex $x$
and ending at a simplex $y$. The trace $tr(D^{2k})$ is the total number of closed paths in 
$\G$ which have length $2k$.  
\end{propo}
\begin{proof} 
We expand $(d+d^*)^k$. For odd $k$ we have expressions of the form $d d^* d \cdots d d^*$
or $d^* d \cdots d^* d$. For even $k$, we have expressions of the form $d d^* \cdots d^*$ or
$d^* d \cdots d$.  The second statement follows from summing over $x$. 
\end{proof} 

As a consequence of the McKean-Singer theorem we have the following corollary
which is a priori not so obvious because we do not assume any symmetry for the graph. 

\begin{coro}
Let $G$ be an arbitrary finite simple graph. 
The number of closed paths in $\G$ of length $2k$ starting at even dimensional simplices
is equal to the number of closed paths of length $2k$ starting at odd dimensional simplices.  
\end{coro}
\begin{proof} 
$\str(D^{2k}) = \tr(L^k|\Omega_b) - \tr(L^k|\Omega_f) = 0$.
\end{proof}

For example, on a triangle, there are $6$ closed paths of length $4$ starting at a vertex
and $9$ closed paths of length $4$ starting at an edge and $9$ closed paths of length $4$
starting at a triangle. There are $3 \cdot 6 + 9 = 27$ paths starting at an even dimensional 
simplex (vertex or triangle) and  $27=3 \cdot 9$ paths starting at an odd dimensional 
simplex (edge). 

\section{Curvature}

Finally, we want to write the curvature $K(x)$ of a graph using the operator $D$. 
For a vertex $x \in V$, denote by $V_k(x)$ the number of $K_{k+1}$ graphs in the 
unit sphere $S(x)$. The curvature at a vertex $x$ is defined as
$$ K(x) = 1+\sum_{k=1}^{\infty} (-1)^k \frac{V_{k-1}(x)}{k+1} \; .  $$
It satisfies the {\bf Gauss-Bonnet theorem}  \cite{cherngaussbonnet}
$$ \sum_{x \in V} K(x) = \chi(G) \; ,  $$
an identity which holds for any finite simple graph.  The result becomes more interesting and deeper,
if more gometric structure on the graph is assumed. For example, for geometric graphs, where the unit spheres
are discrete spheres of fixed dimension with Euler characteristic like in the continuum, 
then $K(x)=0$ for odd dimensional graphs. This result \cite{indexformula} 
relies on discrete integral geometric methods and in particular on \cite{indexexpectation} which 
assures that curvature is the expectation of the index of functions. 

\begin{defn}
Denote by $L_p(x)$ the operator $L_p$ 
restricted to the unit sphere $S(x)$. It can be thought of as an analogue of a {\bf signature} for 
differential operators.  Define the linear operator 
$A_p \to A_p'(x) = A_{p-1}(x)/(p+1)$ so that $A_p''(x) = A_{p-2}(x)/(p (p+1))$. 
\end{defn}

\begin{coro}
The curvature $K(x)$ at a vertex $x$ of a graph satisfies 
$$  K(x) = \str(L(x)'') \; . $$
\end{coro}
\begin{proof}
From 
$\tr(L_p(x)) = (p+2) V_{p+1}$ we get
$V_{p-1}(x) = \tr(L_{p-2}(x))/p$ and so 
$$  K(x) = \sum_{p=0}^{\infty} \frac{V_{p-1}}{p+1} = \sum \frac{\tr(L_{p-2})}{p (p+1)} = \str(L''(x)) \; . $$
\end{proof}

We have now also in the discrete an operator theoretical interpretation of Gauss-Bonnet-Chern
theorem in the same way as in the continuum, where the 
{\bf Atiyah-Singer index theorem} provides this interpretation. \\

If we restrict $D$ to the even part $D: \Omega_b \to \Omega_f$,
the Euler characteristic is the {\bf analytic index} $\ker(D) - \ker(D^*)$.
The {\bf topological index} of $D$ is $\sum_{x \in X} \str(D^2(x)'')$. 
As in the continuum, the discrete Gauss-Bonnet theorem is an example of an index theorem. 

\section{Dirac isospectral graphs} 

In this section we describe a general way to get Dirac isospectral graphs and 
give example of an isospectral pairs. In the next section, among the examples, an other 
example is given. There is an analogue quest in the continuum for isospectral nonisometric metrics for
all differential forms. As mentioned in \cite{BergerPanorama}, there are various
examples known. Milnor's examples with flat tori lifts to isospectrality on forms for example.

\begin{figure}[H]
\scalebox{0.35}{\includegraphics{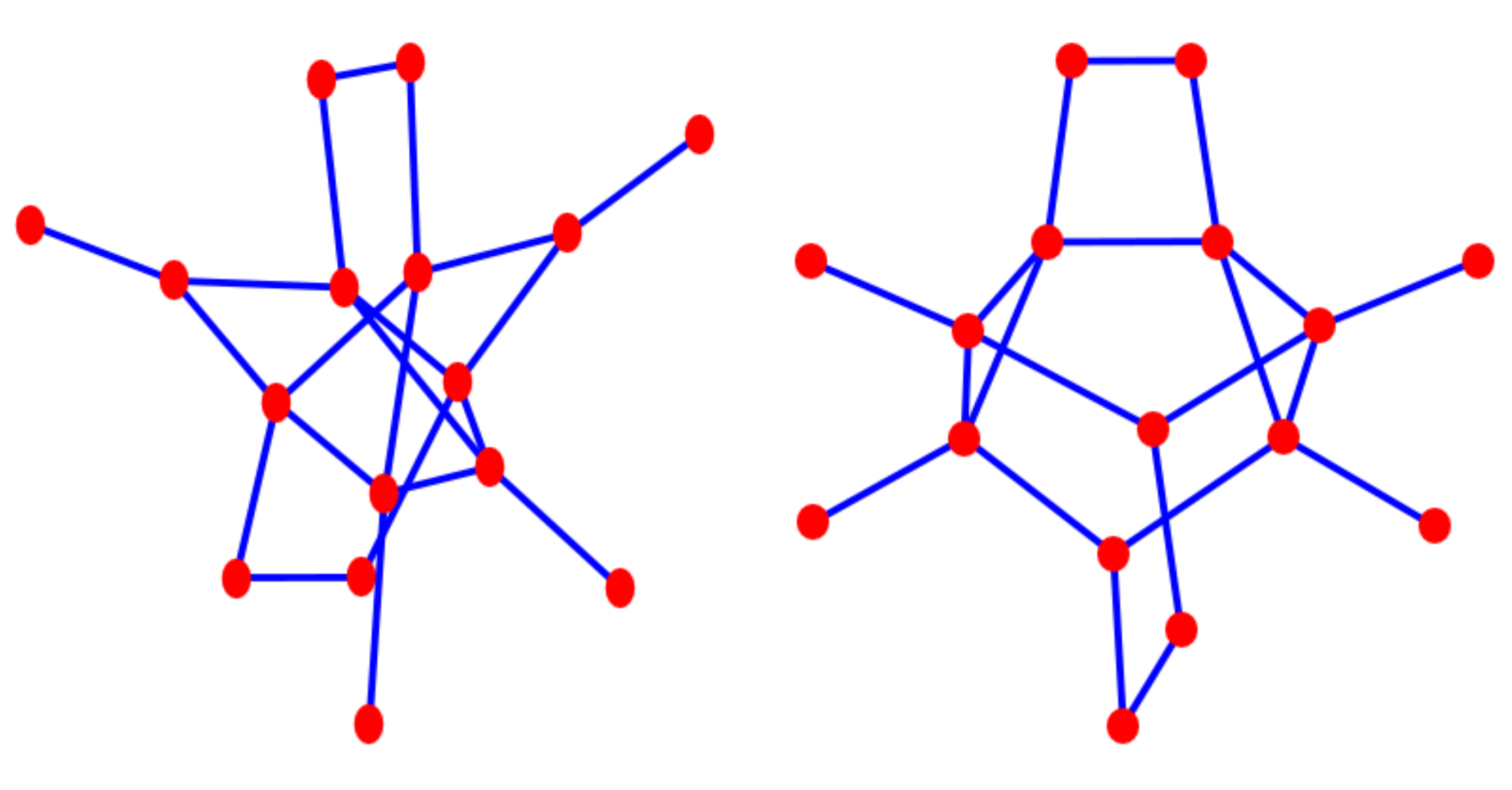}}
\caption{
The Halbeisen-Hungerb\"uhler pair is an example of a pair
of isospectral graphs for the Dirac operator $D$. Both graphs have $v_0=16$ vertices,
$v_1=21$ edges and $v_2=2$ triangles so that $\chi(G)=v_0-v_1+v_2=-3$.
The Betti numbers are $b_0=1,b_1=4,b_2=0$. Since the spectrum of $L_2$ is $\{3,3 \}$
for both graphs and Sunada methods in \cite{HalbeisenHungerbuehler} have shown
that the Laplacians $L_0(G_i)$ on functions are isospectral,
the McKean-Singer theorem in the proof of Proposition~\ref{lifting}
implies that also $L_1$ are isospectral so that the two graphs are isospectral for $D$.
\label{halbeisenhungerbuehler}
}
\end{figure}

\begin{propo}
Two connected finite simple graphs $G_1,G_2$ which have the same number of vertices
and edges and which are triangle free and isospectral with respect to $L_0$
are isospectral with respect to $D$. 
\label{lifting}
\end{propo}
\begin{proof}
Since $\chi(G_i) = v_0-v_1 = b_0-b_1 = 1-b_1$ are the same, Hodge theory shows that
$L_1$ has $b_1$ zero eigenvalues in both cases. Because the graphs are connected,
$L_0$ have $b_0=1$ zero eigenvalues. The McKean-Singer theorem implies that the 
nonzero eigenvalues of $L_0$ match the nonzero eigenvalues of $L_1$. 
\end{proof}

This can be applied also in situations, where triangles are present and where the spectrum 
on $L_2$ is trivially equivalent. 
An example has been found in \cite{HalbeisenHungerbuehler}. 
That paper uses Sunada's techniques to construct isospectral simple graphs. 
We show that such examples can also be isospectral with respect to the Dirac operator. 

\begin{figure}[H]
\scalebox{0.35}{\includegraphics{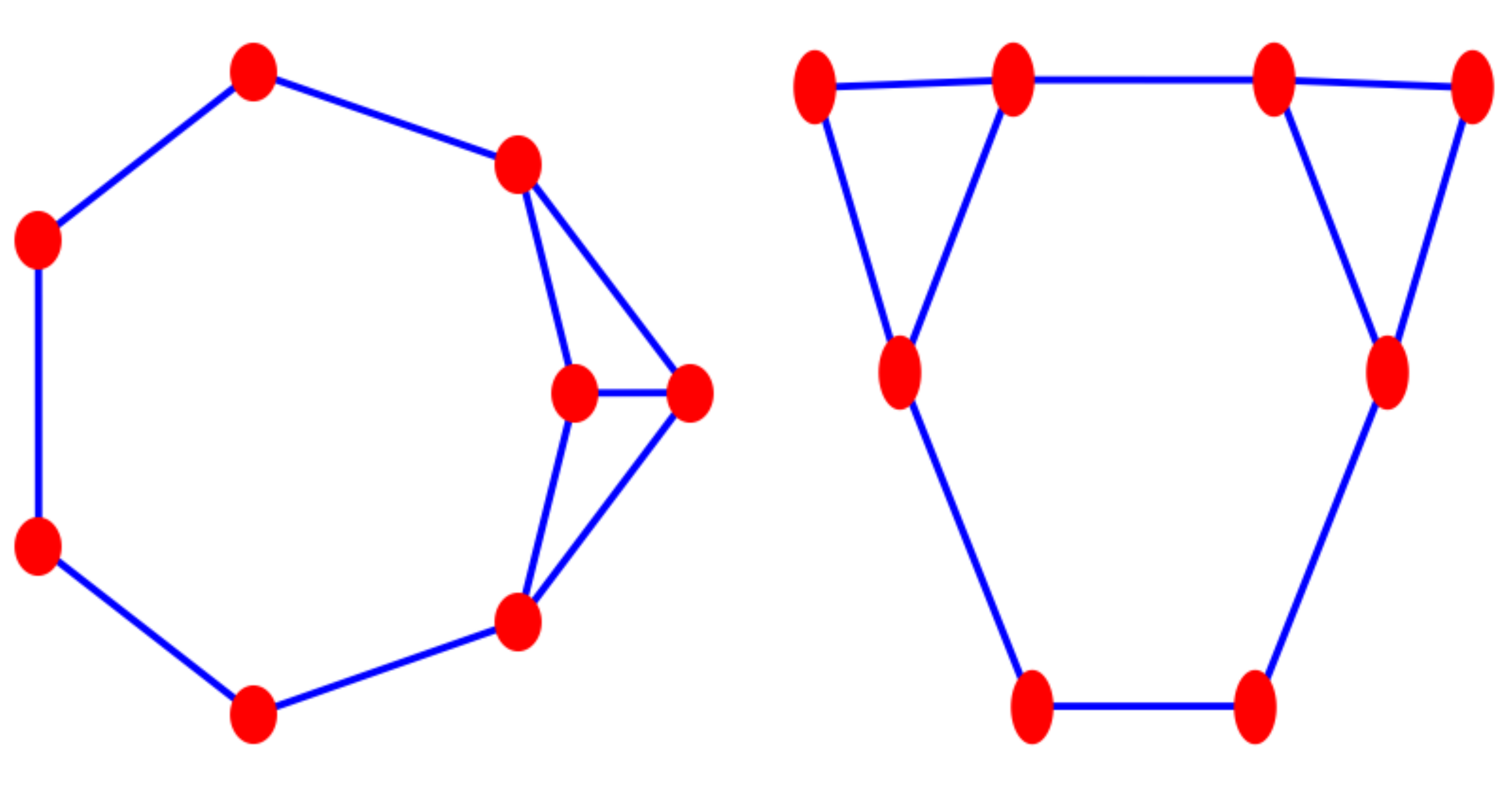}}
\caption{
An example of a cospectral pair for the Laplace operator $L_0$ given by 
Haemers-Spence \cite{HaemersSpence}. Both graphs have $v_0=8$ vertices,
$v_1=10$ edges and $v_2=2$ triangles so that $\chi(G)=v_0-v_1+v_2=0$.
The Betti numbers are $b_0=1,b_1=1,b_2=0$. While the spectra of $L_0$ agree,
the spectra of $L_2$ are not the same. For the left graph, the eigenvalues are 
$4,2$, for the right graph, the eigenvalues are $3$ and $3$. We would not even have
to compute the eigenvalues because $\tr(L_2^2)$, which has a combinatorial interpretation
as paths of length $4$ in $\G$ does not agree. We conclude from McKean-Singer
(without having to compute the eigenvalues) that also the eigenvalues of $L_1$ do not agree. 
\label{haemersspence}
}
\end{figure}

These examples are analogues to examples of manifolds given by Gornet \cite{BergerPanorama}
where isospectrality for functions and differential forms can be different. 

\begin{figure}[H]
\scalebox{0.35}{\includegraphics{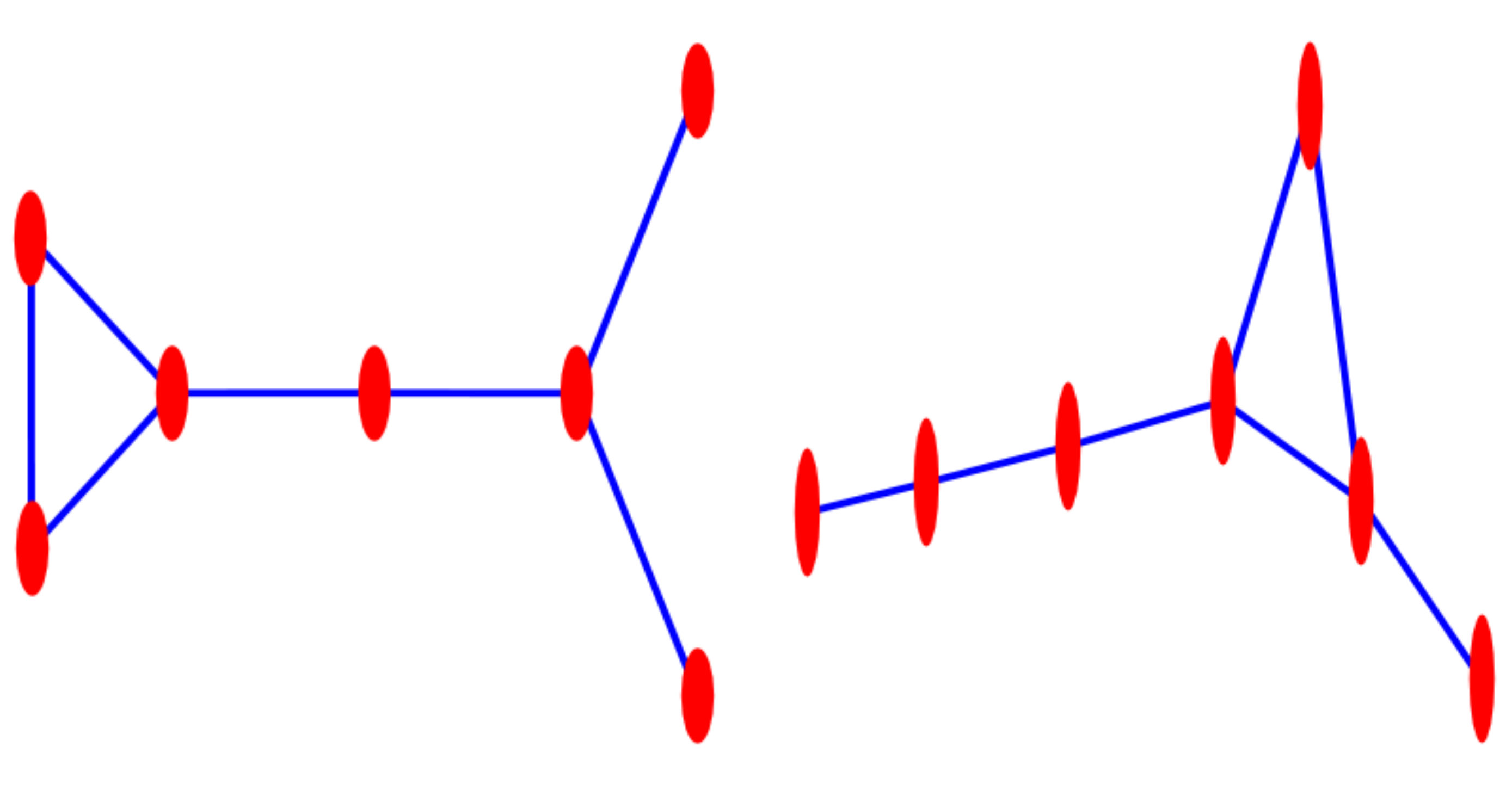}}
\caption{
Two isospectral graphs for $L_0$ found in \cite{Tan}. Both are
contractible and have Euler characteristic $1$. 
Since they have the same operator $L_2$, by McKean Singer 
also the spectra of $L_1$ are the same. The two graphs are
isospectral with respect to $D$. 
\label{tan}
}
\end{figure}

The rest of the article consists of examples. 

\vfill
\pagebreak

\section{Examples}

\subsection{Cyclic graph}
Let $G$ be the cyclic graph with $4$ elements. 

\begin{figure}[H]
\scalebox{0.35}{\includegraphics{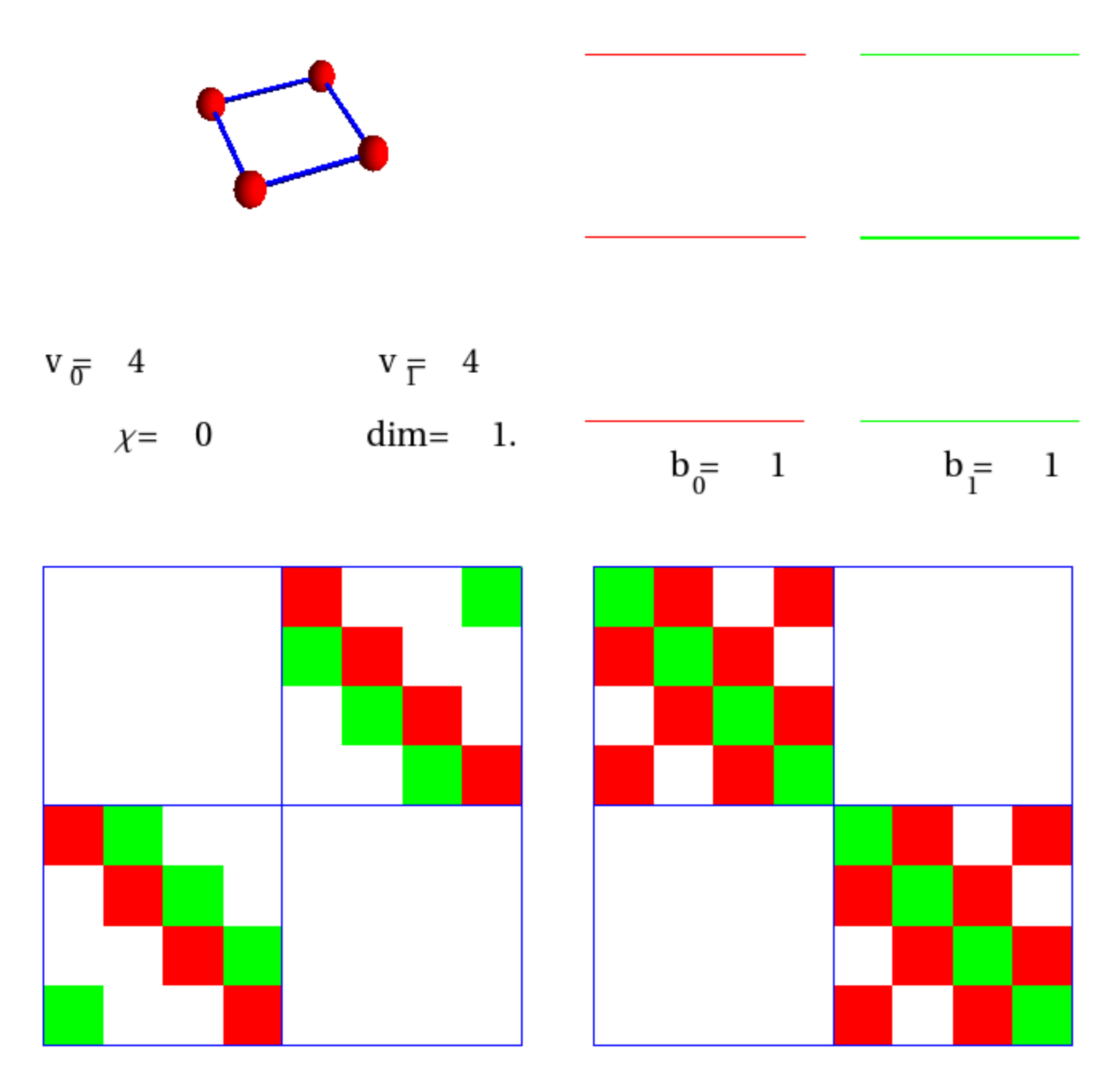}}
\caption{
\label{cyclic}
}
\end{figure}

The exterior derivative $d$
is a map from $\Omega_0=R^4 \to \Omega^1=R^4$. It defines $L_0=d^*d$. 
The operator $L_1: \Omega_1 \to \Omega_1 = d d^*$ is the same than $L_0$ and 
${\rm tr}(\exp(-t L_0) - {\rm tr}(\exp(-t L_1) ) = 0$ for all $t$. We have
$$ L_0 = d_0^* d_0 = \left[
                  \begin{array}{cccc}
                   2 & -1 & 0 & -1 \\
                   -1 & 2 & -1 & 0 \\
                   0 & -1 & 2 & -1 \\
                   -1 & 0 & -1 & 2 \\
                  \end{array}
                  \right],
   L_1 = d_0 d_0^* = \left[
                  \begin{array}{cccc}
                   2 & -1 & 0 & -1 \\
                   -1 & 2 & -1 & 0 \\
                   0 & -1 & 2 & -1 \\
                   -1 & 0 & -1 & 2 \\
                  \end{array}
                  \right] \; . $$
and so
\begin{tiny}
$$D=\left[
     \begin{array}{cccccccc}
      0 & 0 & 0 & 0 & -1 & 0 & 0 & 1 \\
      0 & 0 & 0 & 0 & 1 & -1 & 0 & 0 \\
      0 & 0 & 0 & 0 & 0 & 1 & -1 & 0 \\
      0 & 0 & 0 & 0 & 0 & 0 & 1 & -1 \\
      -1 & 1 & 0 & 0 & 0 & 0 & 0 & 0 \\
      0 & -1 & 1 & 0 & 0 & 0 & 0 & 0 \\
      0 & 0 & -1 & 1 & 0 & 0 & 0 & 0 \\
      1 & 0 & 0 & -1 & 0 & 0 & 0 & 0 \\
     \end{array}
     \right],
L=\left[
     \begin{array}{cccccccc}
      2 & -1 & 0 & -1 & 0 & 0 & 0 & 0 \\
      -1 & 2 & -1 & 0 & 0 & 0 & 0 & 0 \\
      0 & -1 & 2 & -1 & 0 & 0 & 0 & 0 \\
      -1 & 0 & -1 & 2 & 0 & 0 & 0 & 0 \\
      0 & 0 & 0 & 0 & 2 & -1 & 0 & -1 \\
      0 & 0 & 0 & 0 & -1 & 2 & -1 & 0 \\
      0 & 0 & 0 & 0 & 0 & -1 & 2 & -1 \\
      0 & 0 & 0 & 0 & -1 & 0 & -1 & 2 \\
     \end{array}
     \right]  \; . $$
\end{tiny}
The bosonic eigenvalues are $\{4,2,2,0 \; \}$, the fermionic eigenvalues are
$\{4,2,2,0 \; \}$.  The example generalizes to any cyclic graph $C_n$. 
The operators $L_0=L_1$ are Jacobi matrices $2-\tau-\tau^*$ which Fourier 
diagonalizes to the diagonal matrix with entries 
$\lambda_k = 2-2 \cos(2\pi k/n), k,1,\dots ,n$.

\vfill
\pagebreak

\subsection{Triangle}
Let $G$ be the triangle. 

\begin{figure}[H]
\scalebox{0.35}{\includegraphics{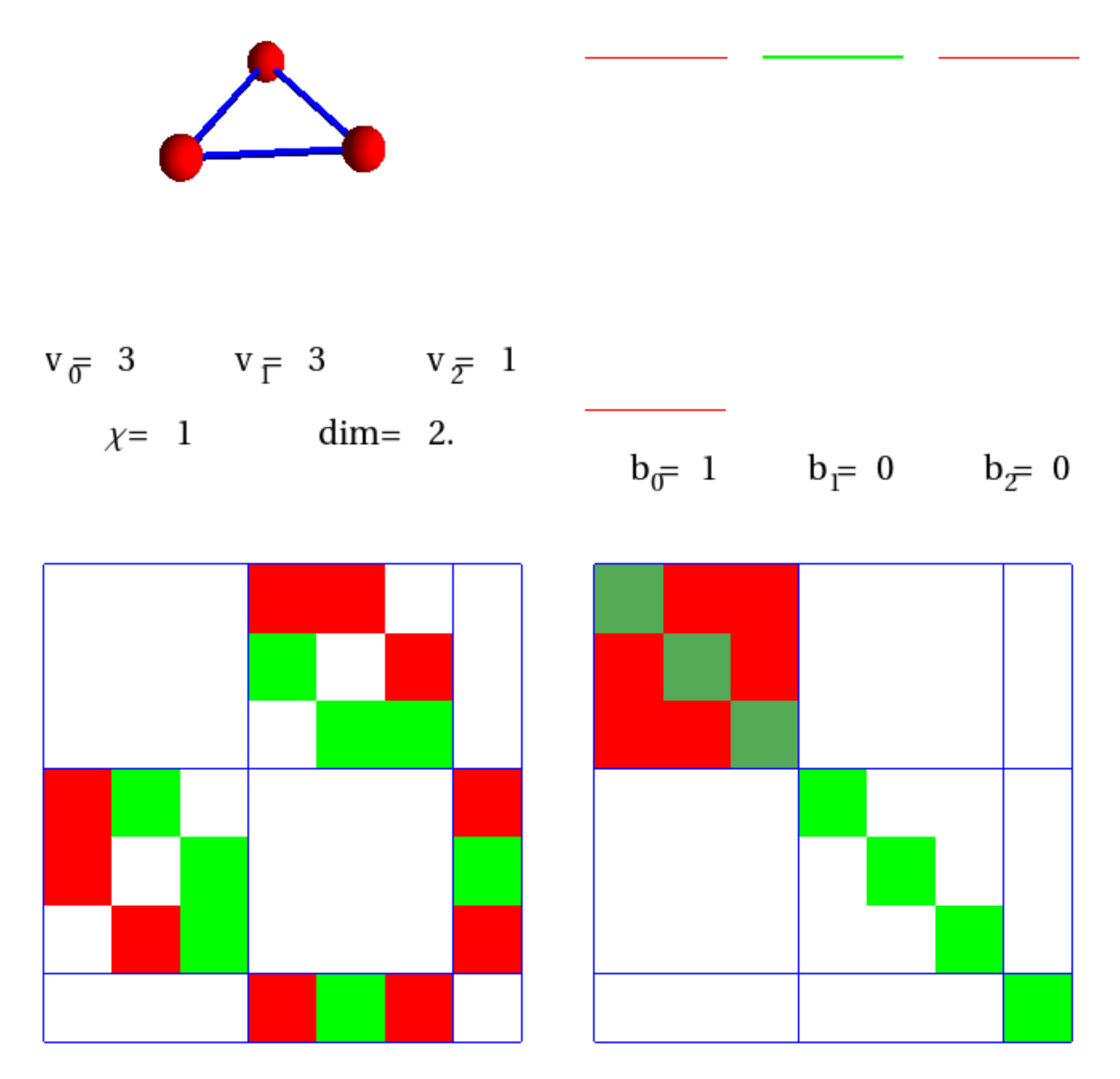}}
\caption{
\label{completegraph4}
}
\end{figure}

The exterior derivative $d: \Omega_0 \to \Omega_1$ is
$$ d_0 = \left[ \begin{array}{ccc|c}
                    1 & 2 & 3 &  \\ \hline
                   -1 & 1 & 0 & a\\
                   1 & 0 & -1 & b\\
                   0 &  1 & 1 & c \\
                  \end{array} \right], 
   d_1 =  \left[ \begin{array}{ccc}
                  -1 & 1 &-1 \\
                  \end{array} \right] \;   $$
and 
$$  L_0=d_0^* d_0 =  \left[ \begin{array}{ccc}
                  2 & -1 & -1 \\
                  -1 & 2 & -1 \\
                  -1 & -1 & 2
                 \end{array} \right], 
    L_1=d_1 d_1^* + d_0^* d_0 = 
                     \left[ \begin{array}{ccc}
                   3 & 0 & 0 \\
                   0 & 3 & 0 \\
                   0 & 0 & 3 \\
                  \end{array} \right] \; 
    L_2 = d_1 d_1^*   =  \left[ \begin{array}{c} 3 \\ \end{array} \right] \; . $$
The Dirac and Laplace operator is 
$$ D=\left[
     \begin{array}{ccccccc}
      0 & 0 & 0 & -1 & -1 & 0 & 0 \\
      0 & 0 & 0 & 1 & 0 & -1 & 0 \\
      0 & 0 & 0 & 0 & 1 & 1 & 0 \\
      -1 & 1 & 0 & 0 & 0 & 0 & -1 \\
      -1 & 0 & 1 & 0 & 0 & 0 & 1 \\
      0 & -1 & 1 & 0 & 0 & 0 & -1 \\
      0 & 0 & 0 & -1 & 1 & -1 & 0 \\
     \end{array}
     \right], \; 
  L=\left[
     \begin{array}{ccccccc}
      2 & -1 & -1 & 0 & 0 & 0 & 0 \\
      -1 & 2 & -1 & 0 & 0 & 0 & 0 \\
      -1 & -1 & 2 & 0 & 0 & 0 & 0 \\
      0 & 0 & 0 & 3 & 0 & 0 & 0 \\
      0 & 0 & 0 & 0 & 3 & 0 & 0 \\
      0 & 0 & 0 & 0 & 0 & 3 & 0 \\
      0 & 0 & 0 & 0 & 0 & 0 & 3 \\
     \end{array}
     \right] \; .  $$
The later has the bosonic eigenvalues $\{3,3,0,3 \; \}$ and
the fermionic eigenvalues $\{3,3,3 \;\}$.  We have 
$$ {\rm tr}(e^{-t L_0}) = e^{-3 t} \left(e^{3 t}+2\right), 
   {\rm tr}(e^{-t L_1}) = 3 e^{-3 t}, 
   {\rm tr}(e^{-t L_2}) = e^{-3 t} ; $$
and so $\str(e^{t L}) = 1$.  \\

\vfill
\pagebreak

\subsection{Tetrahedron}
For the tetrahedron $K_4$, the exterior derivative $d_0$ 
is a map from $\Omega_0=R^4 \to \Omega^1 = R^6$. 

\begin{figure}[H]
\scalebox{0.25}{\includegraphics{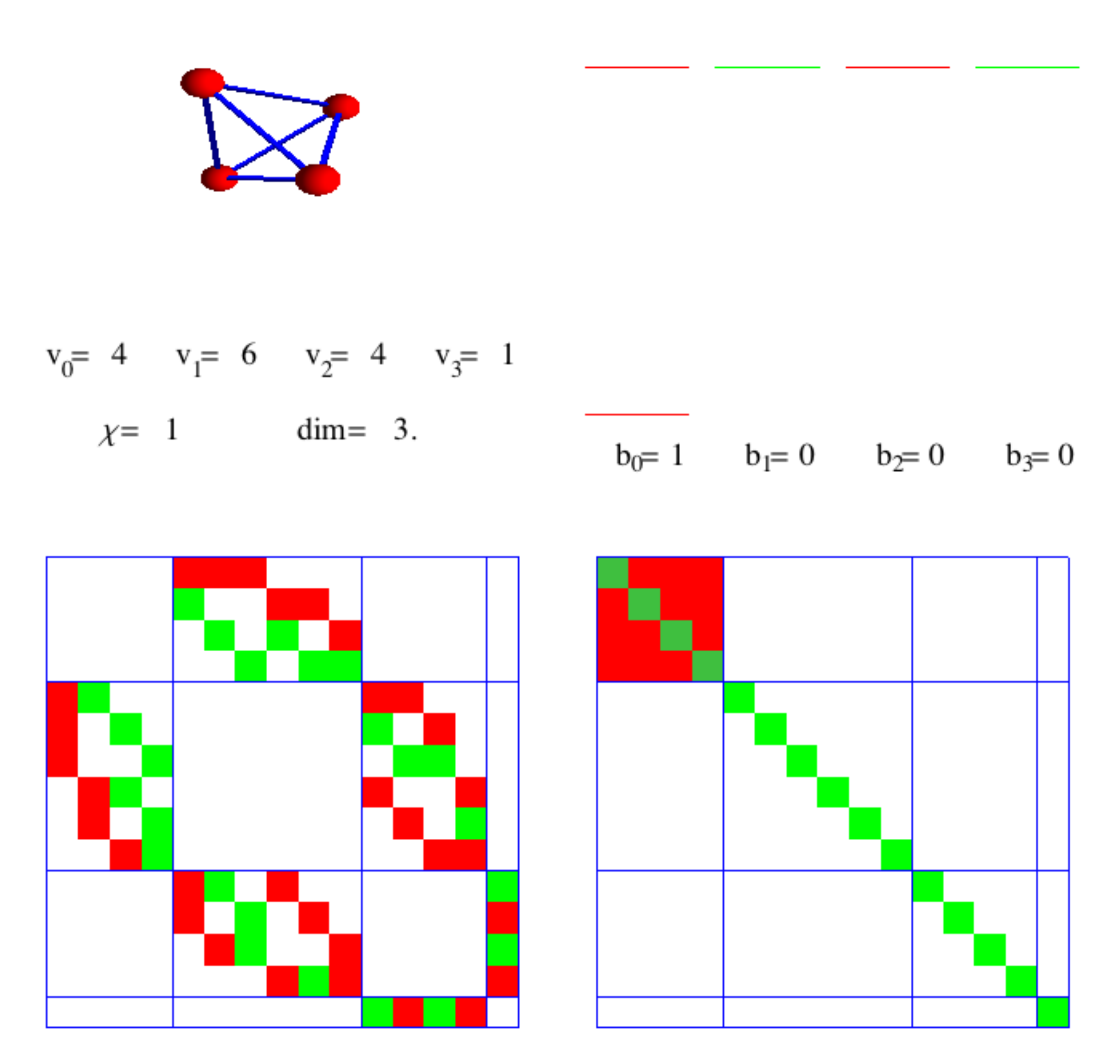}}
\scalebox{0.20}{\includegraphics{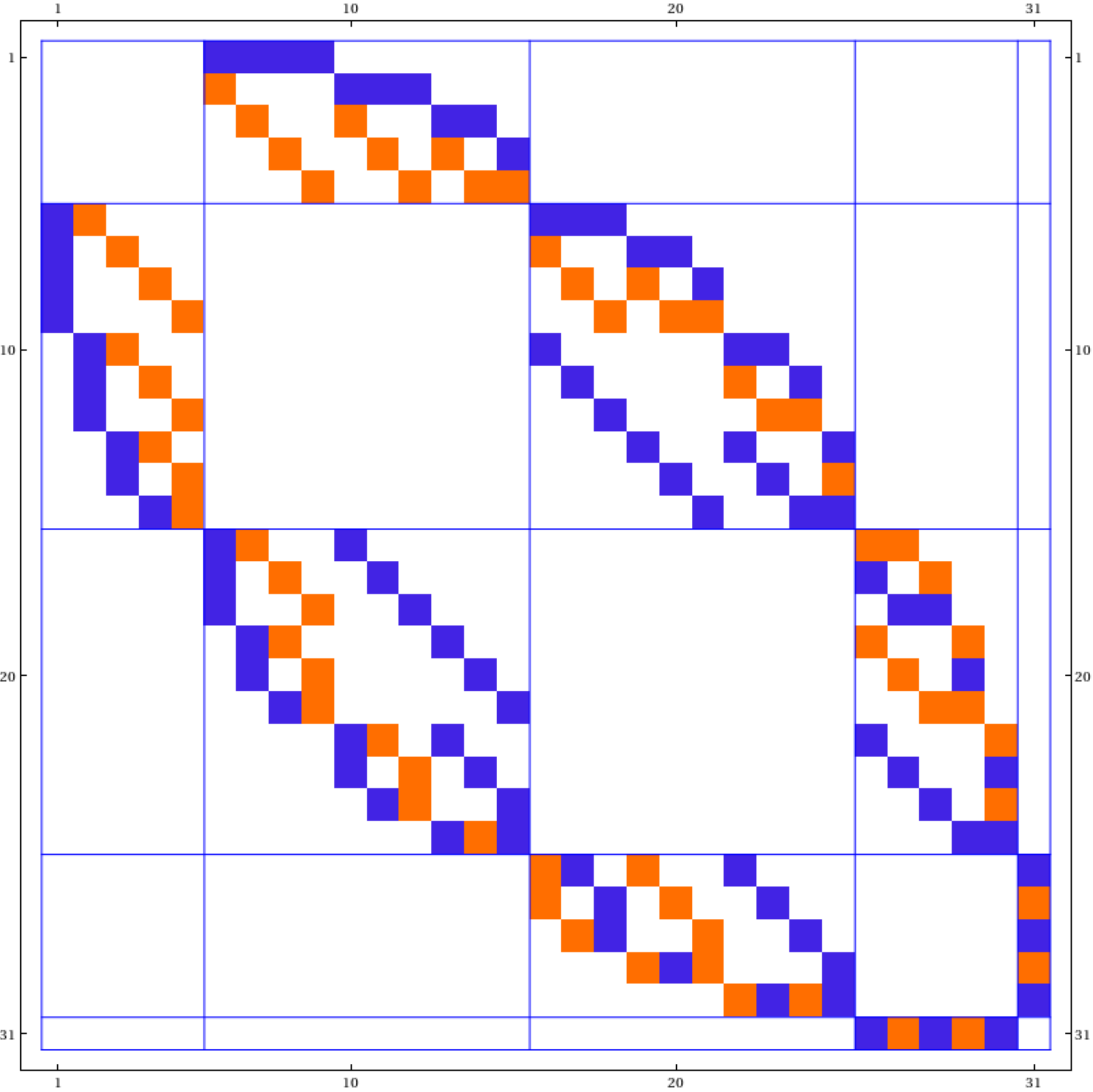}}
\caption{
The Dirac operator of the tetrahedron $K_4$ is shown to the left. 
To the right, we see the Dirac matrix of the hyper 
tetrahedron $K_5$. 
\label{completegraph4}
}
\end{figure}

$d_0=\left[ \begin{array}{cccc}
                   -1 & 1 & 0 & 0 \\
                   -1 & 0 & 1 & 0 \\
                   -1 & 0 & 0 & 1 \\
                   0 & -1 & 1 & 0 \\
                   0 & -1 & 0 & 1 \\
                   0 & 0 & -1 & 1 \end{array} \right]$ and $L_0 = d_0^* d_0=
\left[ \begin{array}{cccc}
                   3 & -1 & -1 & -1 \\
                   -1 & 3 & -1 & -1 \\
                   -1 & -1 & 3 & -1 \\
                   -1 & -1 & -1 & 3
                  \end{array} \right]$.
$d_1=\left[ \begin{array}{cccccc}
                   1 & -1 & 0 & 1 & 0 & 0 \\
                   1 & 0 & -1 & 0 & 1 & 0 \\
                   0 & 1 & -1 & 0 & 0 & 1 \\
                   0 & 0 & 0 & 1 & -1 & 1 \end{array} \right]$ and
$d_1^* d_1 = \left[ \begin{array}{cccccc}
                   2 & -1 & -1 & 1 & 1 & 0 \\
                   -1 & 2 & -1 & -1 & 0 & 1 \\
                   -1 & -1 & 2 & 0 & -1 & -1 \\
                   1 & -1 & 0 & 2 & -1 & 1 \\
                   1 & 0 & -1 & -1 & 2 & -1 \\
                   0 & 1 & -1 & 1 & -1 & 2 \end{array} \right]$ and
$d_0 d_0^* = \left[ \begin{array}{cccccc}
                   2 & 1 & 1 & -1 & -1 & 0 \\
                   1 & 2 & 1 & 1 & 0 & -1 \\
                   1 & 1 & 2 & 0 & 1 & 1 \\
                   -1 & 1 & 0 & 2 & 1 & -1 \\
                   -1 & 0 & 1 & 1 & 2 & 1 \\
                   0 & -1 & 1 & -1 & 1 & 2 \end{array} \right]$ which sums up
to $L_1 = d_1^* d_1 + d_0 d_0^* = 4 Id$. Finally, $d_2 = [1,1,1,1]$ and 
$L_2= d_2^* d_2 + d_1 d_1^* =\left[ \begin{array}{cccc}
                   4 & 0 & 0 & 0 \\
                   0 & 4 & 0 & 0 \\
                   0 & 0 & 4 & 0 \\
                   0 & 0 & 0 & 4 \end{array} \right]$. 
For a complete graph $K_{n}$, the eigenvalues of $L_0$ are $\{ 0,n, \dots, n \; \}$. 
The other operators $L_k$ with $k \geq 1$ are diagonal with entries $n$. The 
Dirac operator $D$ restricted to $p$-forms is a $B(n+1,p+1) \times B(n+1,p)$ matrix,
where $B(n,k)=n!/(k! (n-k)!)$ is the binomial coefficient. 

\vfill \pagebreak

\subsection{Octahedron} 

Lets just write out the matrix $L=D^2=$

\begin{figure}[h]
\scalebox{0.35}{\includegraphics{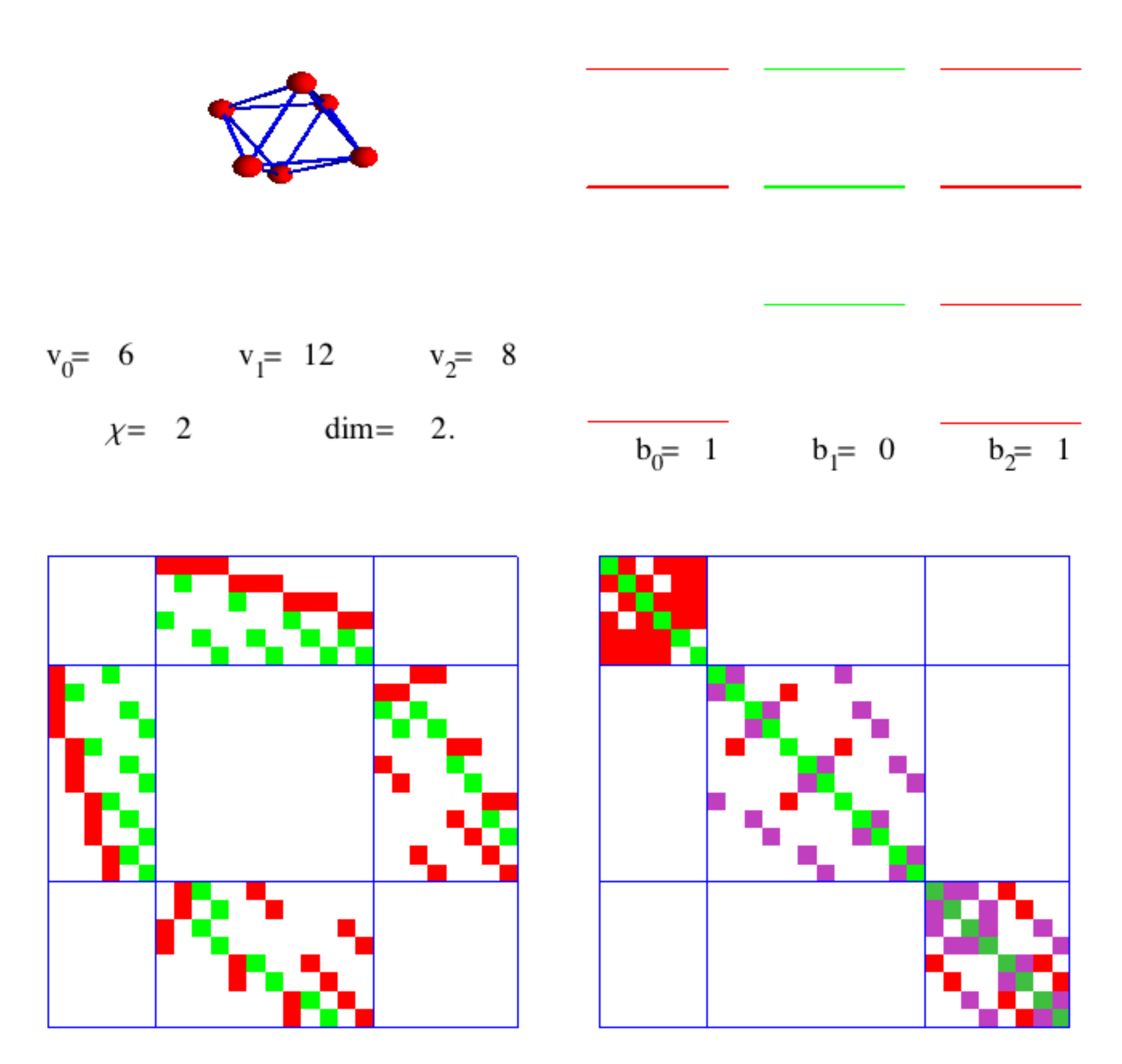}}
\caption{
\label{octahedron}
}
\end{figure}

\begin{tiny}
$$ \left[ \begin{array}{cccccccccccccccccccccccccc}
4&-1&0&-1&-1&-1&0&0&0&0&0&0&0&0&0&0&0&0&0&0&0&0&0&0&0&0\\
-1&4&-1&0&-1&-1&0&0&0&0&0&0&0&0&0&0&0&0&0&0&0&0&0&0&0&0\\
0&-1&4&-1&-1&-1&0&0&0&0&0&0&0&0&0&0&0&0&0&0&0&0&0&0&0&0\\
-1&0&-1&4&-1&-1&0&0&0&0&0&0&0&0&0&0&0&0&0&0&0&0&0&0&0&0\\
-1&-1&-1&-1&4&0&0&0&0&0&0&0&0&0&0&0&0&0&0&0&0&0&0&0&0&0\\
-1&-1&-1&-1&0&4&0&0&0&0&0&0&0&0&0&0&0&0&0&0&0&0&0&0&0&0\\
0&0&0&0&0&0&4&1&0&0&0&0&0&1&0&0&0&0&0&0&0&0&0&0&0&0\\
0&0&0&0&0&0&1&4&0&0&-1&0&0&0&0&0&0&0&0&0&0&0&0&0&0&0\\
0&0&0&0&0&0&0&0&4&1&0&0&0&0&1&0&0&0&0&0&0&0&0&0&0&0\\
0&0&0&0&0&0&0&0&1&4&0&0&0&0&0&1&0&0&0&0&0&0&0&0&0&0\\
0&0&0&0&0&0&0&-1&0&0&4&0&0&-1&0&0&0&0&0&0&0&0&0&0&0&0\\
0&0&0&0&0&0&0&0&0&0&0&4&1&0&0&0&1&0&0&0&0&0&0&0&0&0\\
0&0&0&0&0&0&0&0&0&0&0&1&4&0&0&0&0&1&0&0&0&0&0&0&0&0\\
0&0&0&0&0&0&1&0&0&0&-1&0&0&4&0&0&0&0&0&0&0&0&0&0&0&0\\ 
0&0&0&0&0&0&0&0&1&0&0&0&0&0&4&1&0&0&0&0&0&0&0&0&0&0\\ 
0&0&0&0&0&0&0&0&0&1&0&0&0&0&1&4&0&0&0&0&0&0&0&0&0&0\\ 
0&0&0&0&0&0&0&0&0&0&0&1&0&0&0&0&4&1&0&0&0&0&0&0&0&0\\ 
0&0&0&0&0&0&0&0&0&0&0&0&1&0&0&0&1&4&0&0&0&0&0&0&0&0\\ 
0&0&0&0&0&0&0&0&0&0&0&0&0&0&0&0&0&0&3&1&1&0&-1&0&0&0\\
0&0&0&0&0&0&0&0&0&0&0&0&0&0&0&0&0&0&1&3&0&1&0&-1&0&0\\
0&0&0&0&0&0&0&0&0&0&0&0&0&0&0&0&0&0&1&0&3&1&0&0&1&0\\ 
0&0&0&0&0&0&0&0&0&0&0&0&0&0&0&0&0&0&0&1&1&3&0&0&0&1\\ 
0&0&0&0&0&0&0&0&0&0&0&0&0&0&0&0&0&0&-1&0&0&0&3&1&-1&0\\ 
0&0&0&0&0&0&0&0&0&0&0&0&0&0&0&0&0&0&0&-1&0&0&1&3&0&-1\\ 
0&0&0&0&0&0&0&0&0&0&0&0&0&0&0&0&0&0&0&0&1&0&-1&0&3&1\\
0&0&0&0&0&0&0&0&0&0&0&0&0&0&0&0&0&0&0&0&0&1&0&-1&1&3
\end{array} \right]  \; . $$
\end{tiny}
\begin{tabular}{lll}
${\rm tr}(-t L_0)$  &=  $e^{-6 t} \left(3 e^{2 t}+e^{6 t}+2\right)$  & $\sigma(L_0) = \{ 0, 4, 4, 4, 6, 6 \}$  \\
${\rm tr}(-t L_1)$  &=  $3 e^{-6 t} \left(e^{2 t}+1\right)^2$        & $\sigma(L_1) = \{ 2, 2, 2, 4, 4, 4, 4, 4, 4, 6, 6, 6 \}$ \\
${\rm tr}(-t L_2)$  &=  $e^{-6 t} \left(e^{2 t}+1\right)^3$          & $\sigma(L_2) = \{ 0, 2, 2, 2, 4, 4, 4, 6 \}$ 
\end{tabular}  \\
gives the super trace 
$$ \str(-t L) = {\rm tr}(-t L_0) - {\rm tr}(-t L_1) + {\rm tr}(-t L_2) = 2 \; . $$            

\vfill \pagebreak

\subsection{Icosahedron} 

\begin{figure}[h]
\scalebox{0.32}{\includegraphics{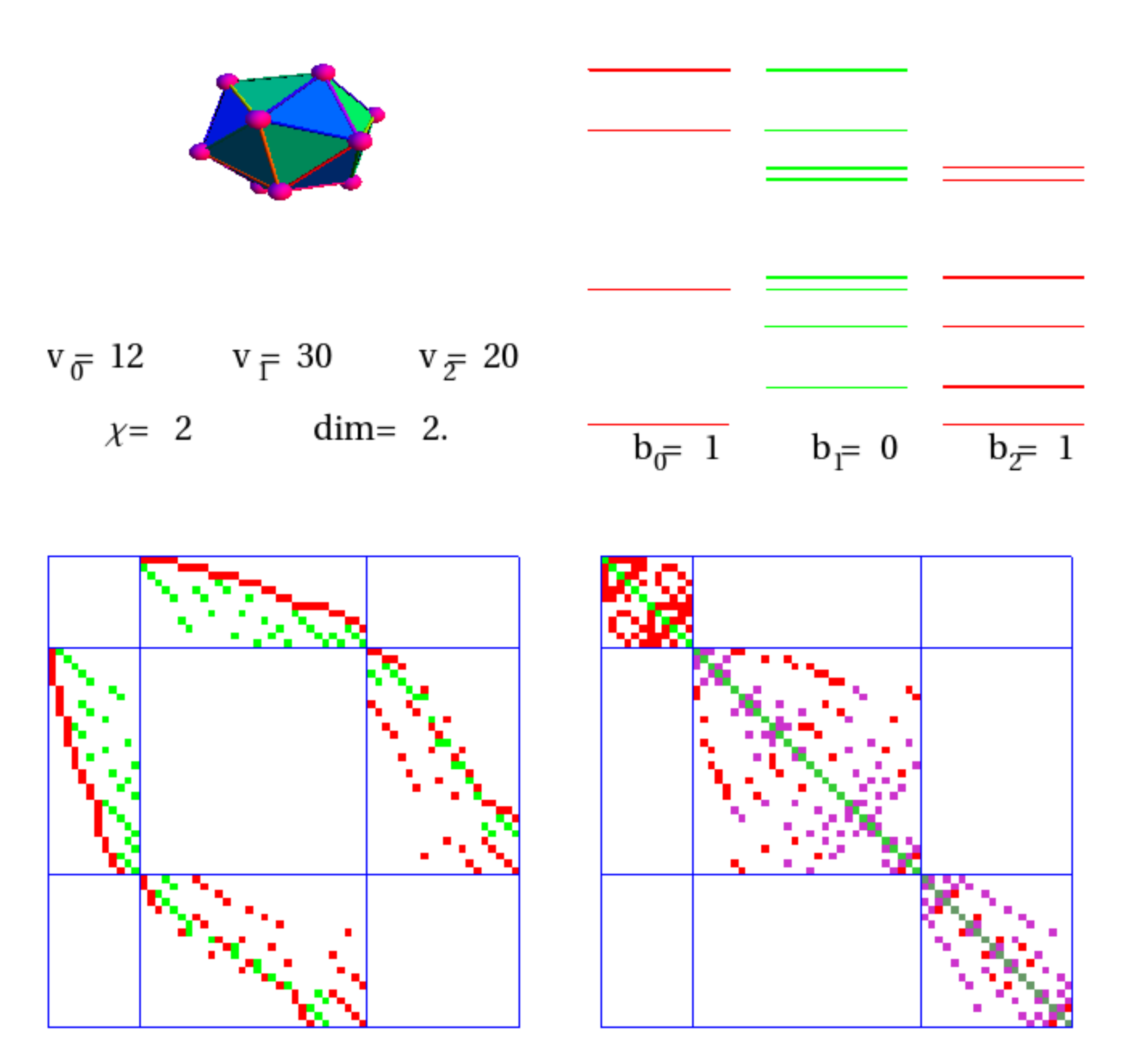}}
\caption{
\label{icosahedron}
}
\end{figure}

We write $\lambda^k$ if the eigenvalue $\lambda$ appears with multiplicity $k$. The bosonic eigenvalues of $L=D^2$ are 
$$ 0^2, 2^5, 3^4, 5^4, 6^5, (3 - \sqrt{5})^3, (5 - \sqrt{5})^3, (3 + \sqrt{5})^3, (5 + \sqrt{5})^3  \; . $$
The fermionic eigenvalues of $L$ are
$$      2^5, 3^4, 5^4, 6^5, (3 - \sqrt{5})^3, (5 - \sqrt{5})^3, (3 + \sqrt{5})^3, (5 + \sqrt{5})^3  \; . $$
The eigenvalues $0$ appear only on bosonic parts matching that the Betti vector $\beta=(1,0,1)$ has its support
on the bosonic part. The picture colors the vertices,edges and triangles according to the ground states, the 
eigenvectors to the lowest eigenvalues of $L_0,L_1$ and $L_2$. 

\begin{figure}[h]
\scalebox{0.32}{\includegraphics{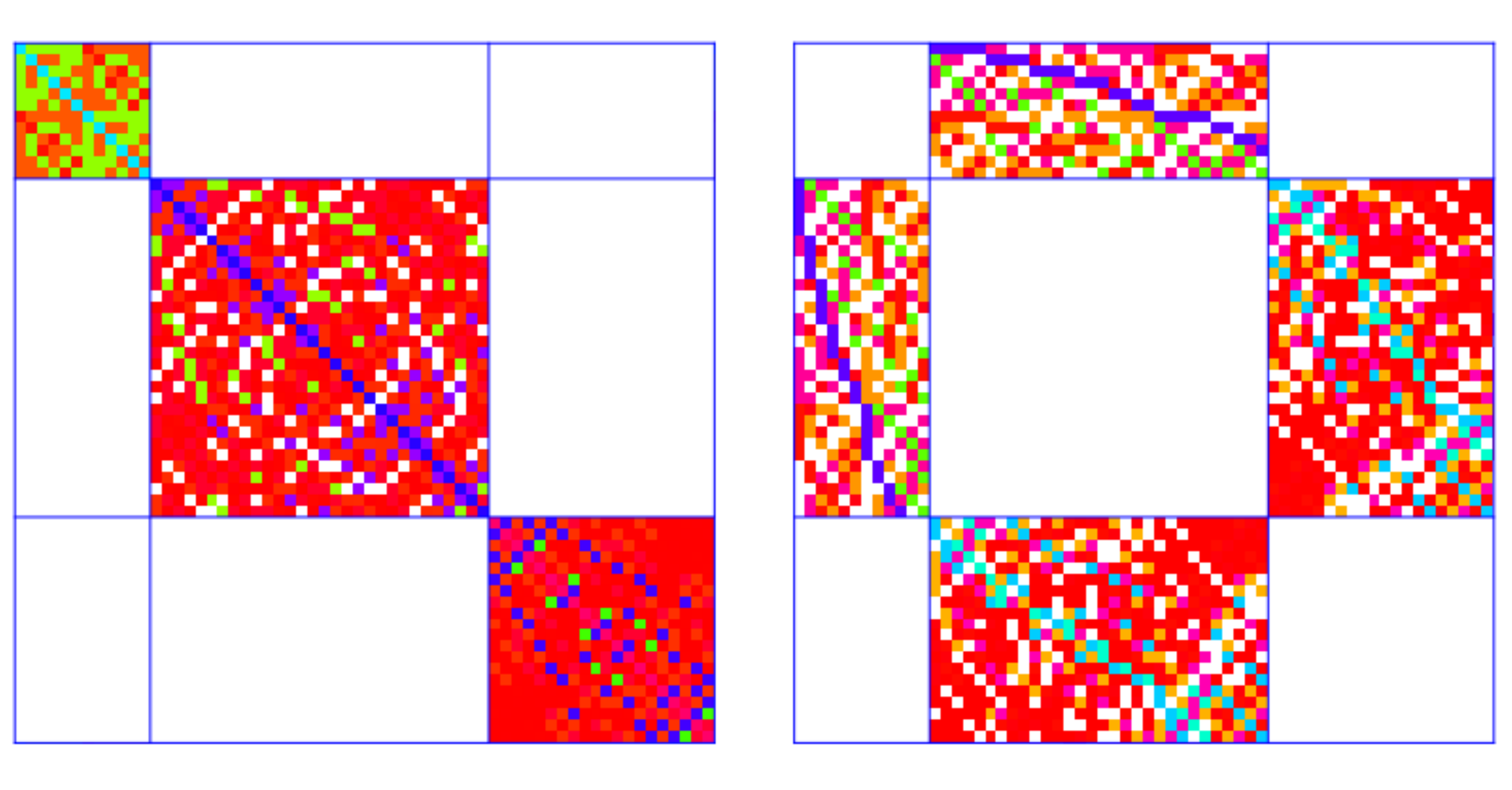}}
\caption{
\label{evolution}
The matrices ${\rm Re}(\exp(i D)) = \cos(D)$ and ${\rm Im}(\exp(i D)) =\sin(D)$ which 
appear in the solution of the wave equation on the graph. Having the matrix $D$ in the computer
makes it easy to watch the wave evolution on a graph. 
}
\end{figure}

\vfill \pagebreak

\subsection{Fullerene}
The figure shows a fullerene type graph based on an icosahedron.  It has 
$v_0=42$ vertices, $v_1=120$ edges and $v_2=80$ triangles. 
The Dirac operator $D$ is a $242 \times 242$ matrix

\begin{figure}[h]
\scalebox{0.35}{\includegraphics{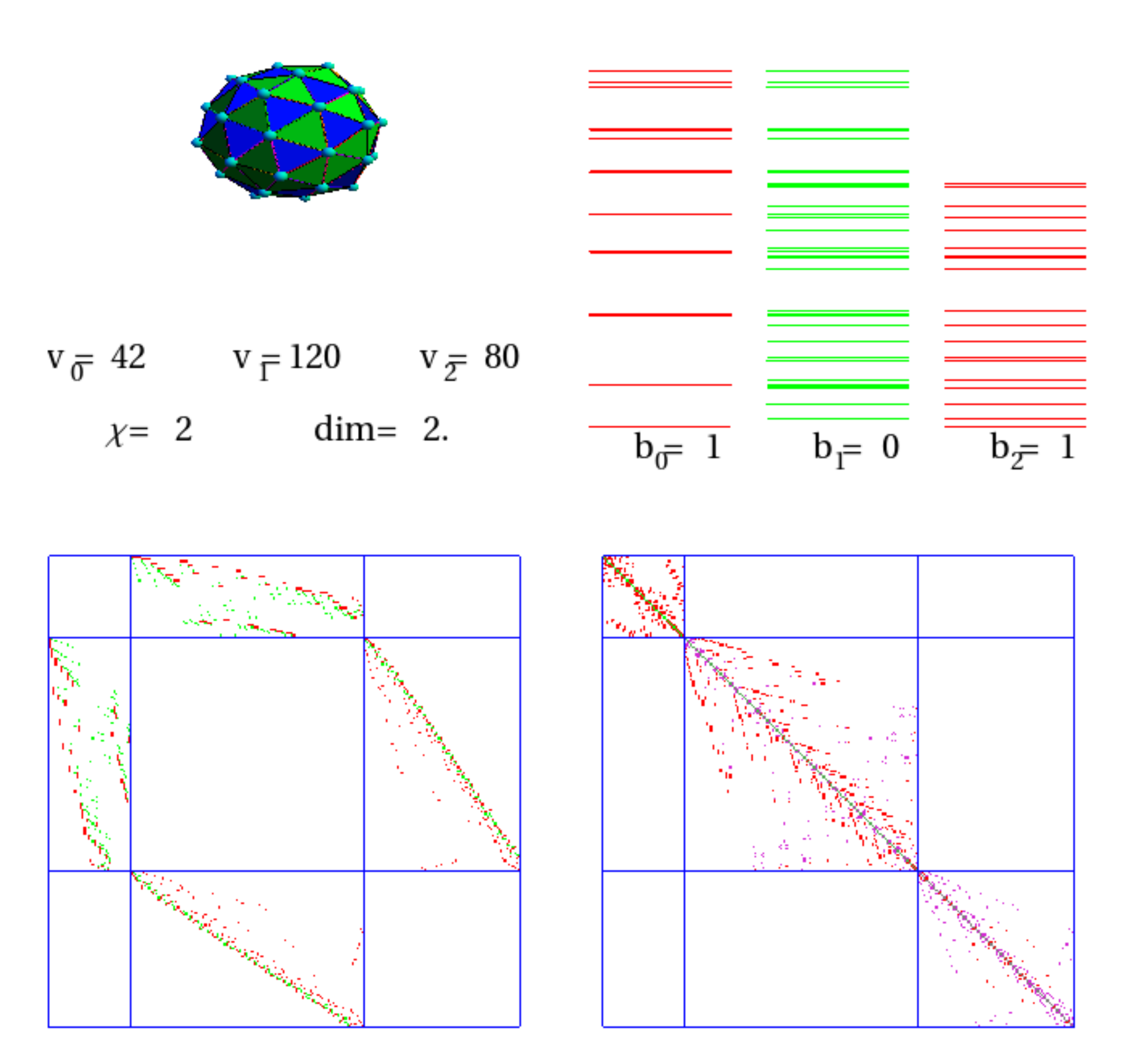}}
\caption{
\label{icosahedron2}
}
\end{figure}

In the next picture we see the three sorted spectra of $L_0,L_1$ and $L_2$. 

\begin{figure}[h]
\scalebox{0.45}{\includegraphics{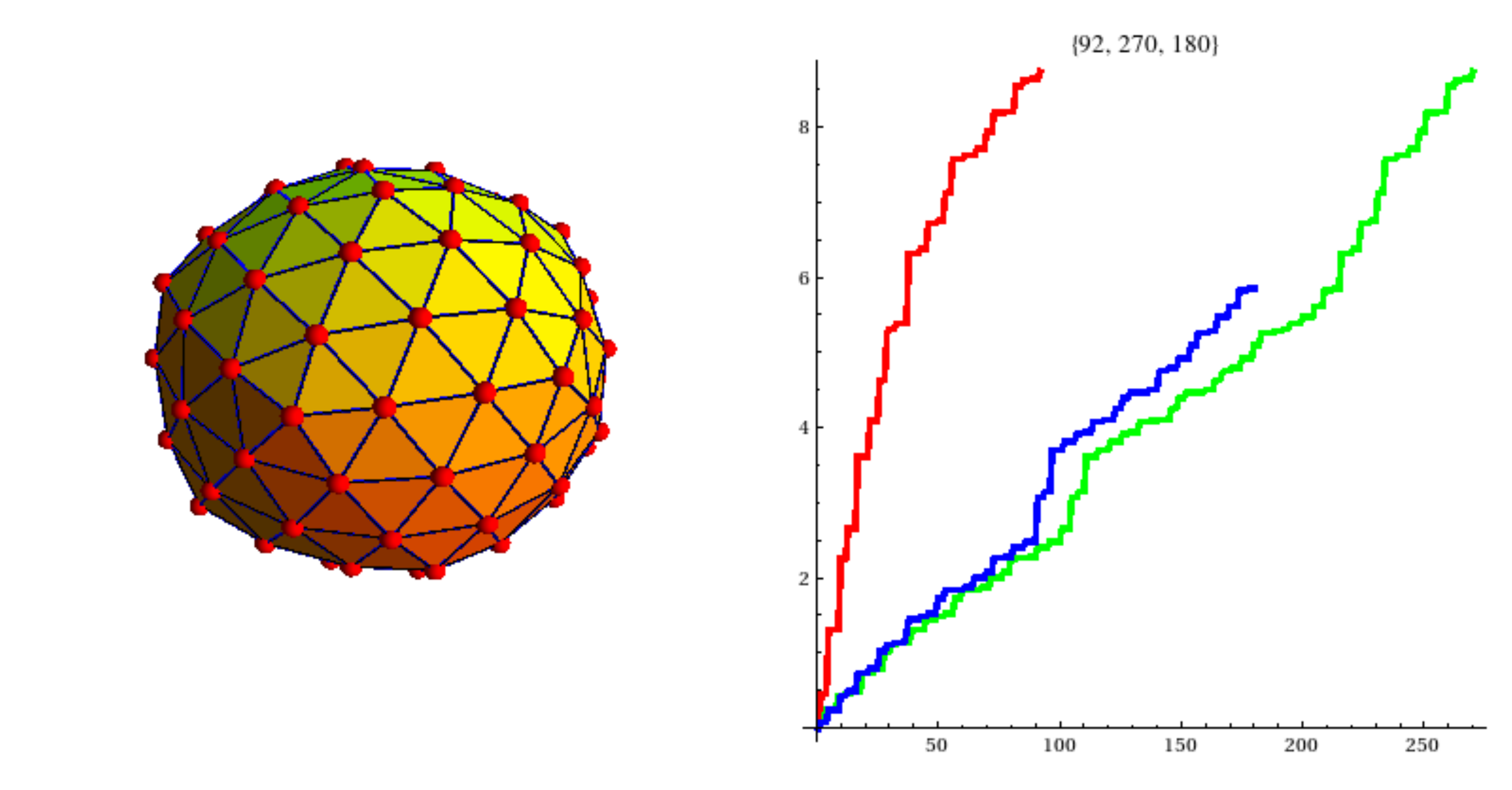}}
\caption{
\label{fullerene4}
}
\end{figure}

\vfill \pagebreak

\subsection{Torus} 

Here is the spectrum of a two dimensional discrete torus

\begin{figure}[h]
\scalebox{0.35}{\includegraphics{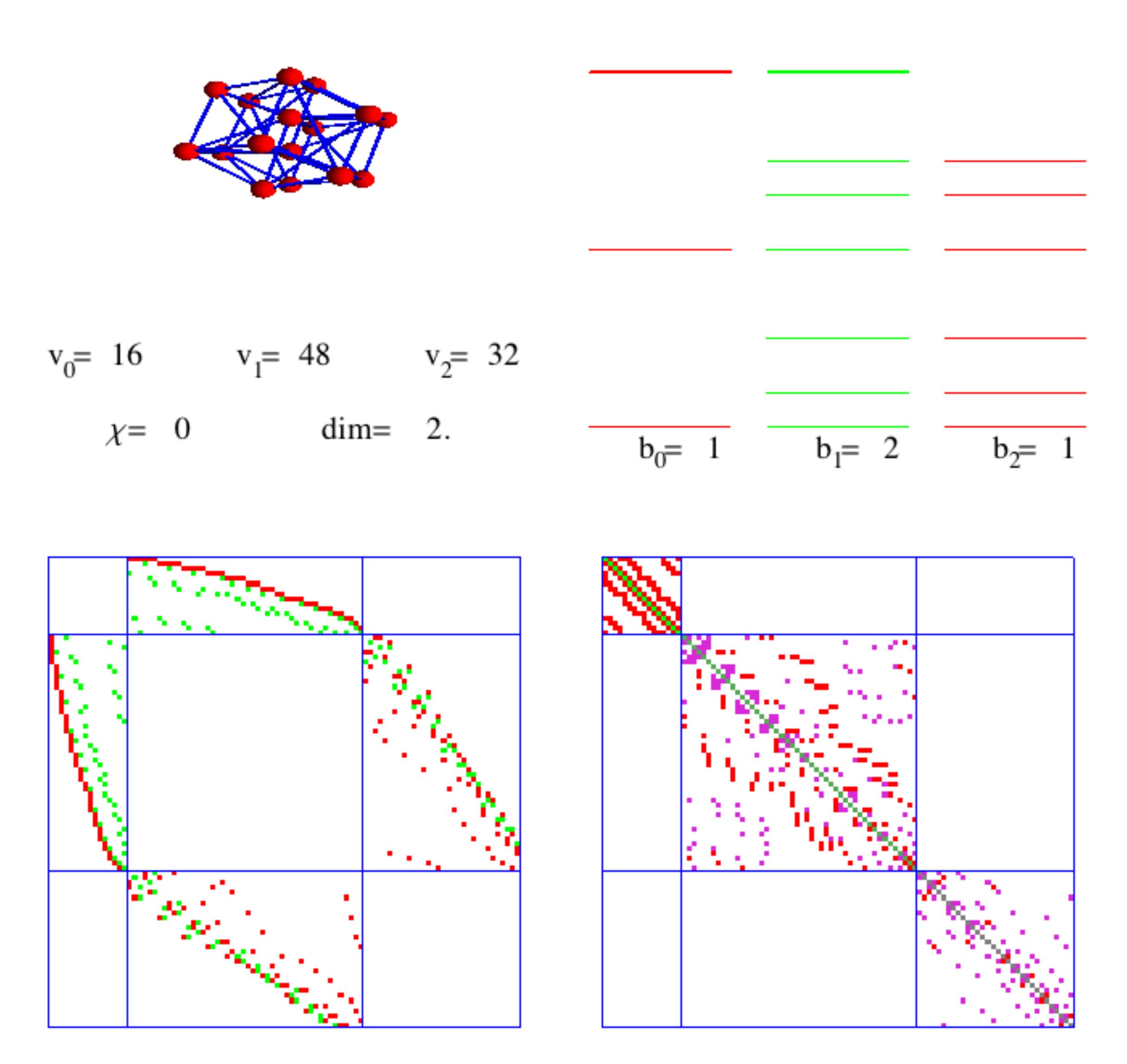}}
\caption{
\label{torus}
}
\end{figure}

and a discrete Klein bottle: 

\begin{figure}[h]
\scalebox{0.35}{\includegraphics{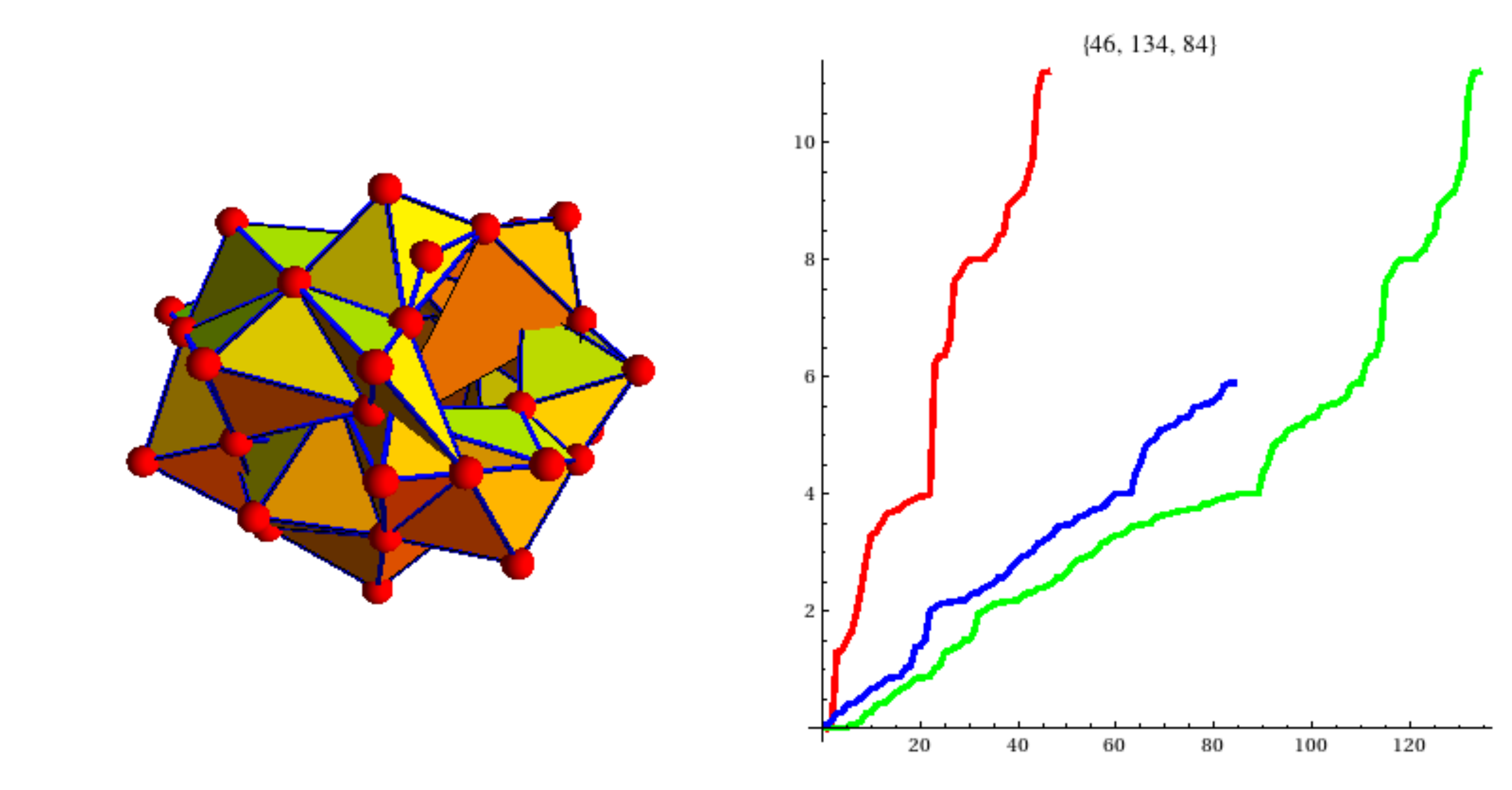}}
\caption{
\label{klein}
}
\end{figure}

\vfill \pagebreak

\pagebreak

\subsection{Dunce hat and Petersen}

The dunce hat is an example important in homotopy theory.
It is a non-contractible graph which is homotopic to a one point graph.

\begin{figure}[h]
\scalebox{0.3}{\includegraphics{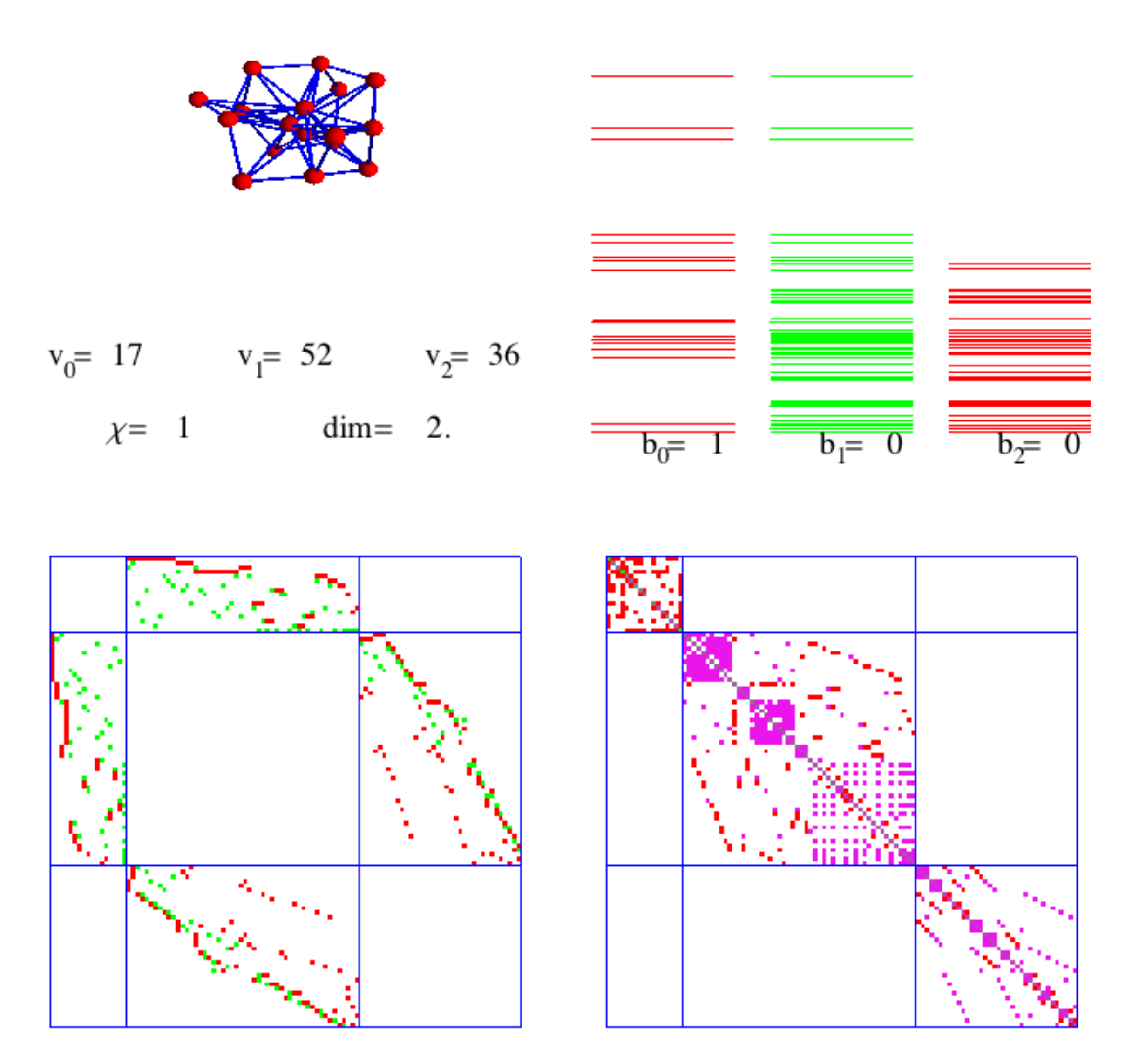}}
\caption{
\label{dunce hat}
}
\end{figure}

Petersen graph

\begin{figure}[h]
\scalebox{0.3}{\includegraphics{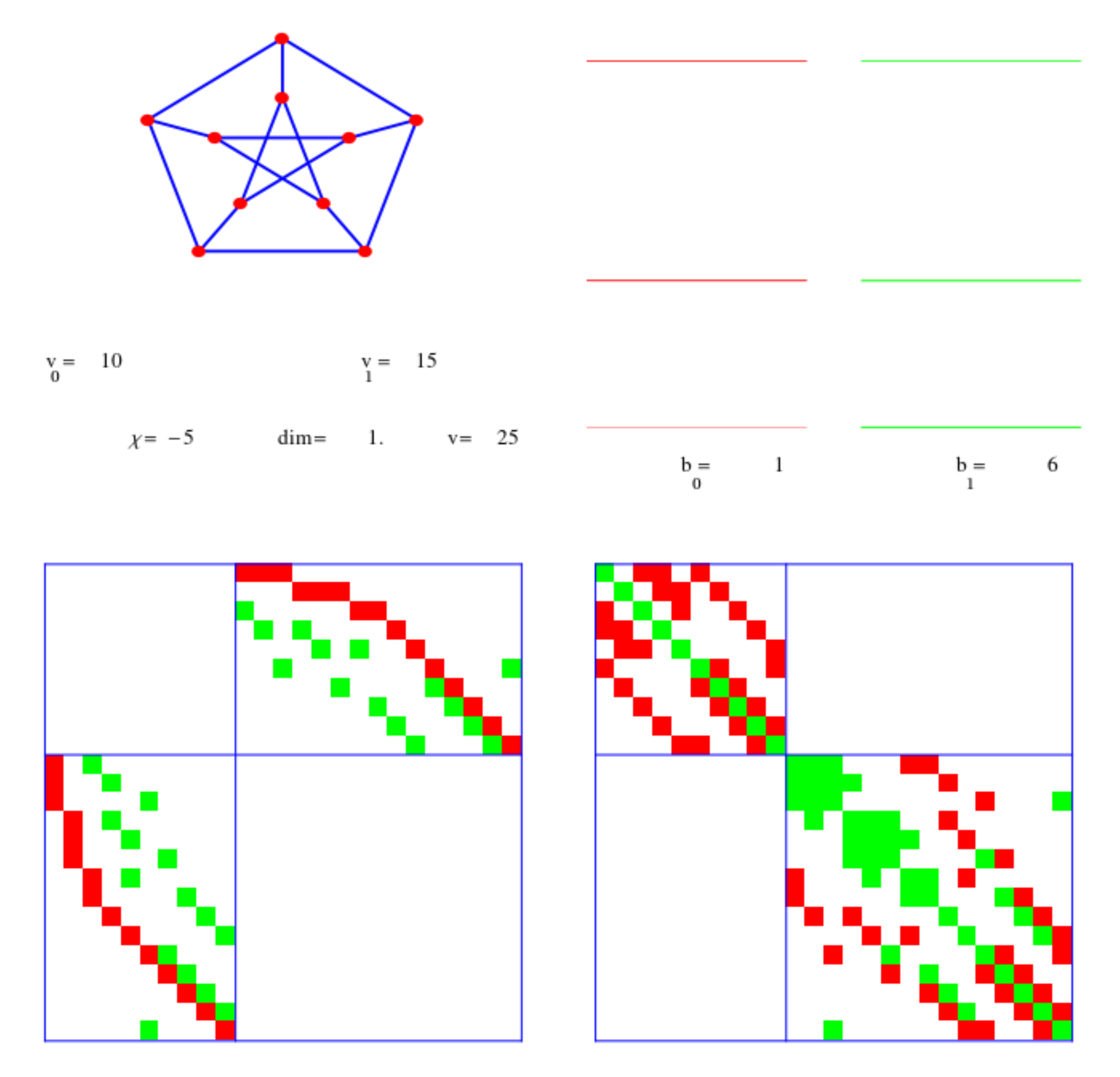}}
\caption{
\label{petersen}
}
\end{figure}

\vfill \pagebreak

\subsection{Isospectral graphs}  

As mentioned in the previous section,
some known isospectral graphs for the graph Laplacian are also
isospectral for the Dirac operator. 

\begin{figure}[ht]
\scalebox{0.30}{\includegraphics{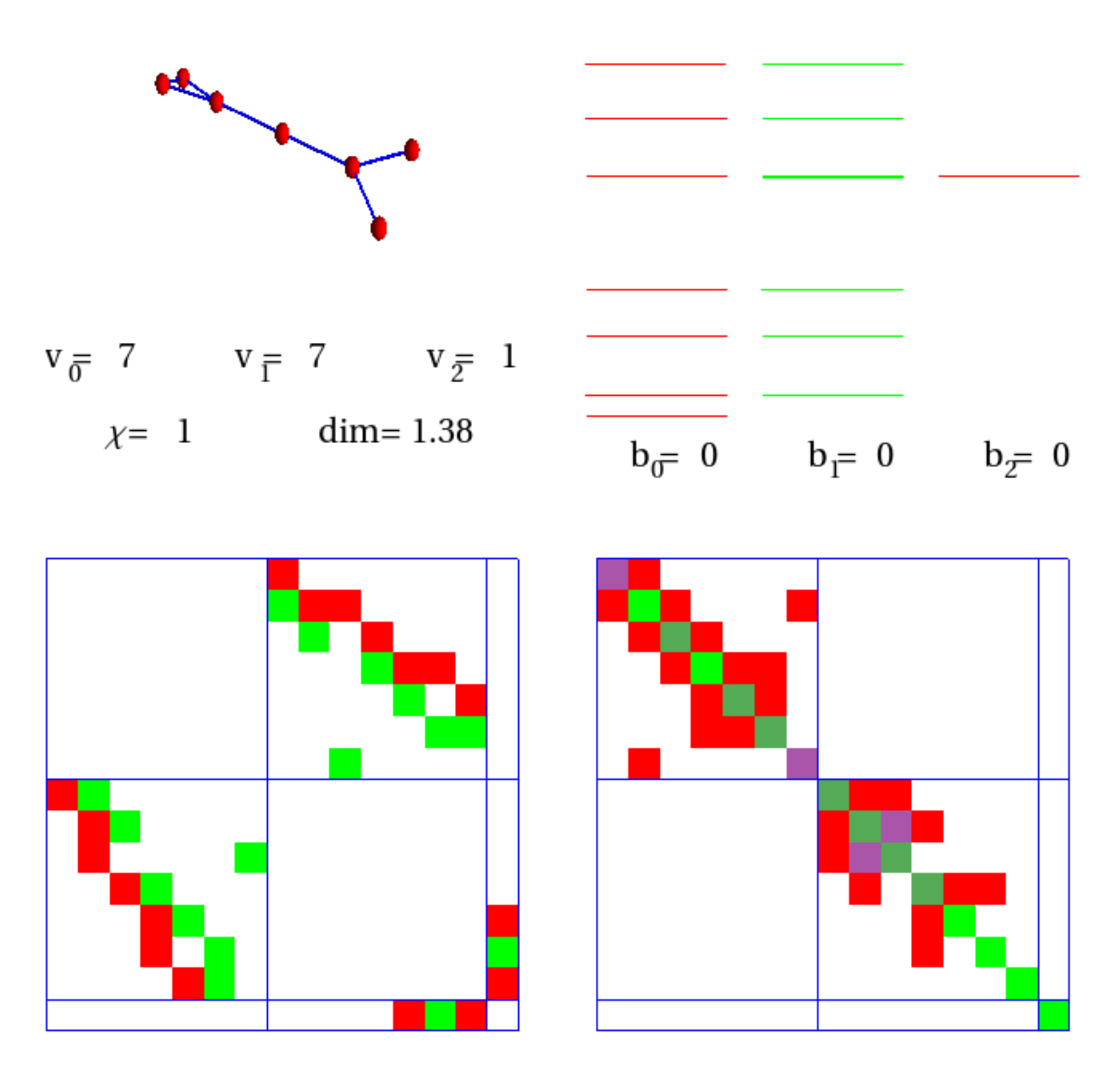}}
\scalebox{0.30}{\includegraphics{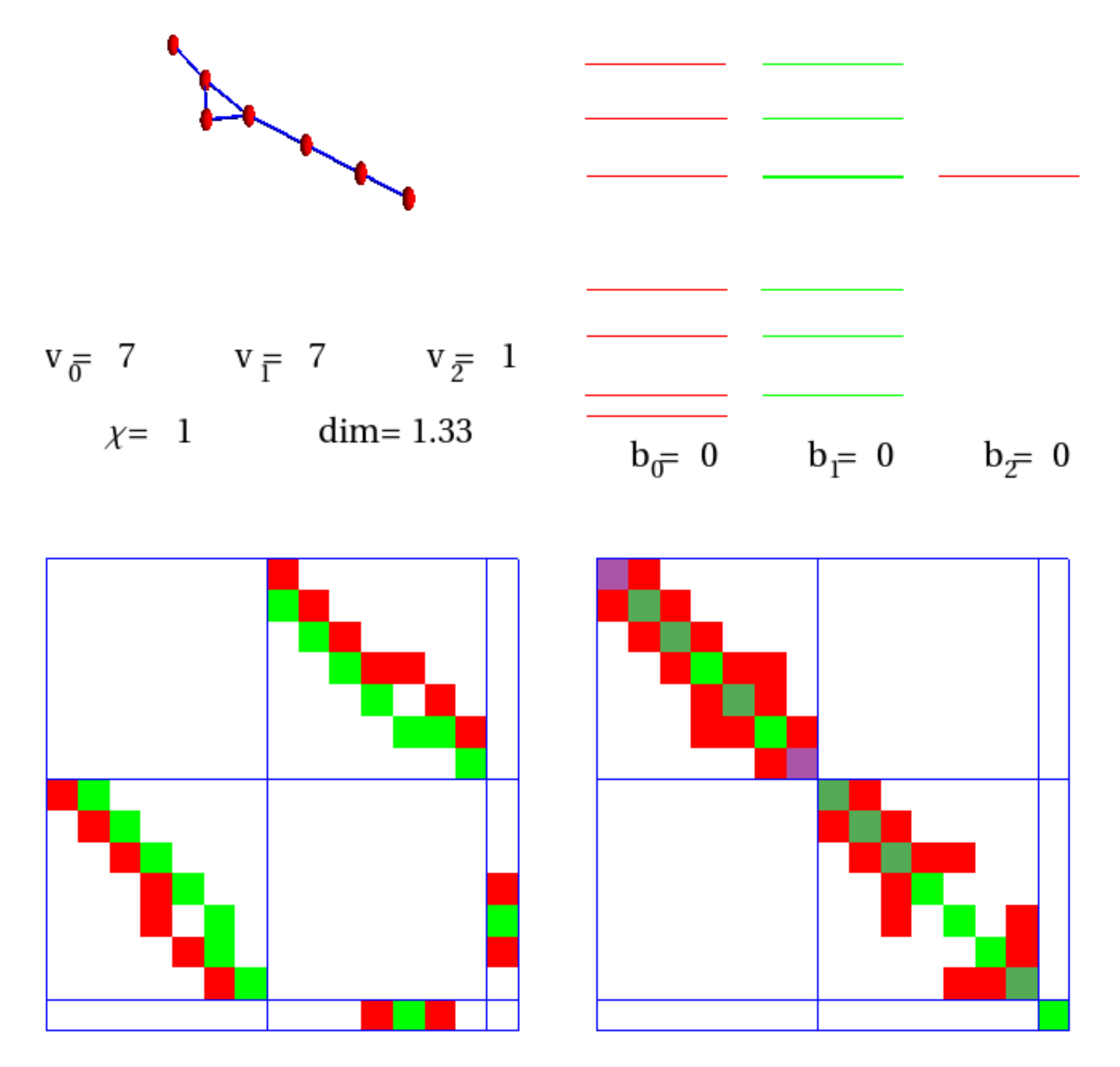}}
\caption{
\label{isospectral1}
}
\end{figure}

In \cite{Tan}, all isospectral connected graphs up
to order $v_0=7$ have been computed. For order $v_0=7$, there are already examples which
are also isospectral for all $p$-forms and for the Dirac operator.
Some of them lead to Dirac isospectral graphs:  \\
$\sigma(D) = \{-\sqrt{3+\sqrt{2}},\sqrt{3+\sqrt{2}},-\sqrt{2+\sqrt{3}},\sqrt{2+\sqrt{3}},-\sqrt{3},-\sqrt{3}$ \\
$\sqrt{3},\sqrt{3},-\sqrt{3-\sqrt{2}},\sqrt{3-\sqrt{2}},-1,1,-\sqrt{2-\sqrt{3}},\sqrt{2-\sqrt{3}},0 \}$. \\
The Laplacian eigenvalues are
 $\sigma(L_0) = \left\{3+\sqrt{2},2+\sqrt{3},3,3-\sqrt{2},1,2-\sqrt{3},0\right\}$ and
 $\sigma(L_1) = \left\{3+\sqrt{2},2+\sqrt{3},3,3,3-\sqrt{2},1,2-\sqrt{3}\right\}$ as well
as $\sigma(L_2) = \{3\}$.

\vfill \pagebreak

\subsection{Planar regions}

One of the questions which fueled research on spectral theory was whether
one ``can hear the shape of a drum". The discrete analogue question is 
whether the Dirac operator can hear the shape of a discrete hexagonal region.

\begin{figure}[ht]
\scalebox{0.30}{\includegraphics{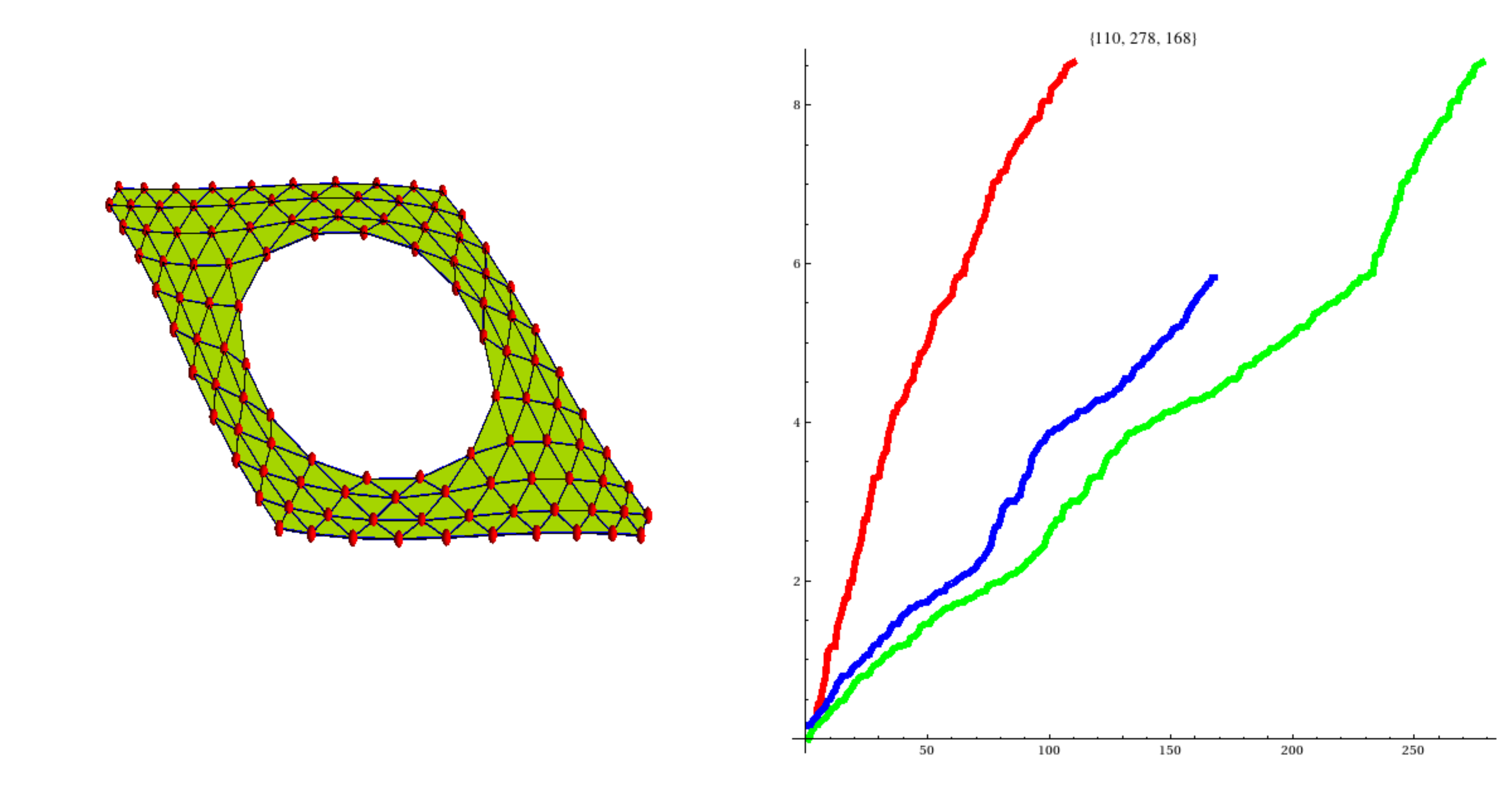}}
\caption{
\label{hexregion}
}
A hexagonal region and the spectra of the Laplacians $L_0,L_1,L_2$. 
\end{figure}

\begin{figure}[ht]
\scalebox{0.30}{\includegraphics{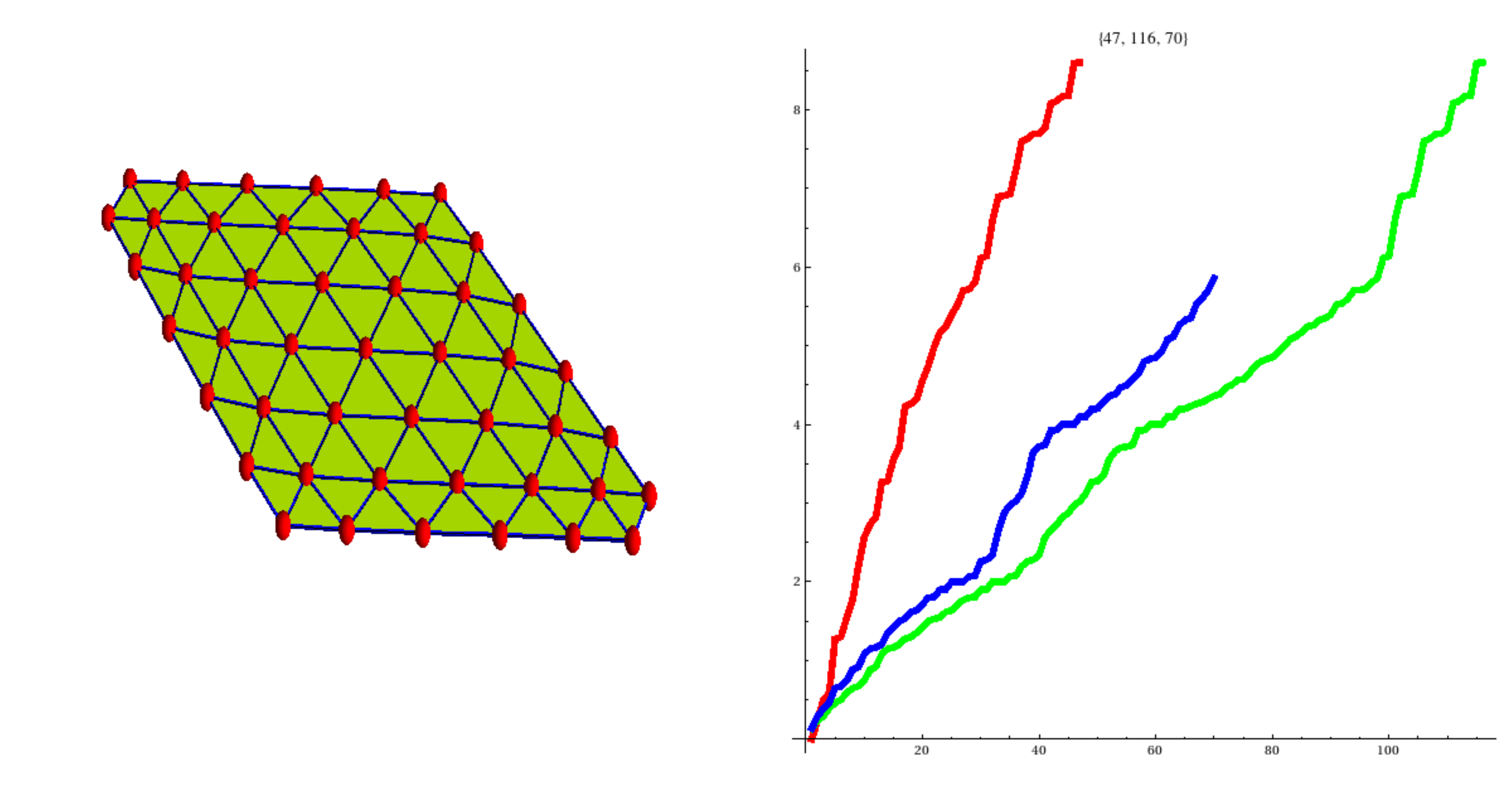}}
\caption{
\label{hexregion}
}
A convex hexagonal region and the spectra of $L_0,L_1,L_2$. 
\end{figure}

As far as we know, no isospectral completely two dimensional 
convex graphs are known. 

\vfill
\pagebreak 

\subsection{Benzenoid}

Isospectral benzenoid graphs were constructed in 
\cite{Babic}. They are isospectral with respect to the adjacency matrix $A$. 
but not with respect to the Laplacian $L=B-A$. After triangulating the hexagons,
the first pair of graphs have Euler characteristic $0$, the second Euler characteristic $1$. 
But after triangularization, the graphs are not isospectral with respect to $A$, nor 
with respect to $L$. 

\begin{figure}[ht]
\scalebox{0.2}{\includegraphics{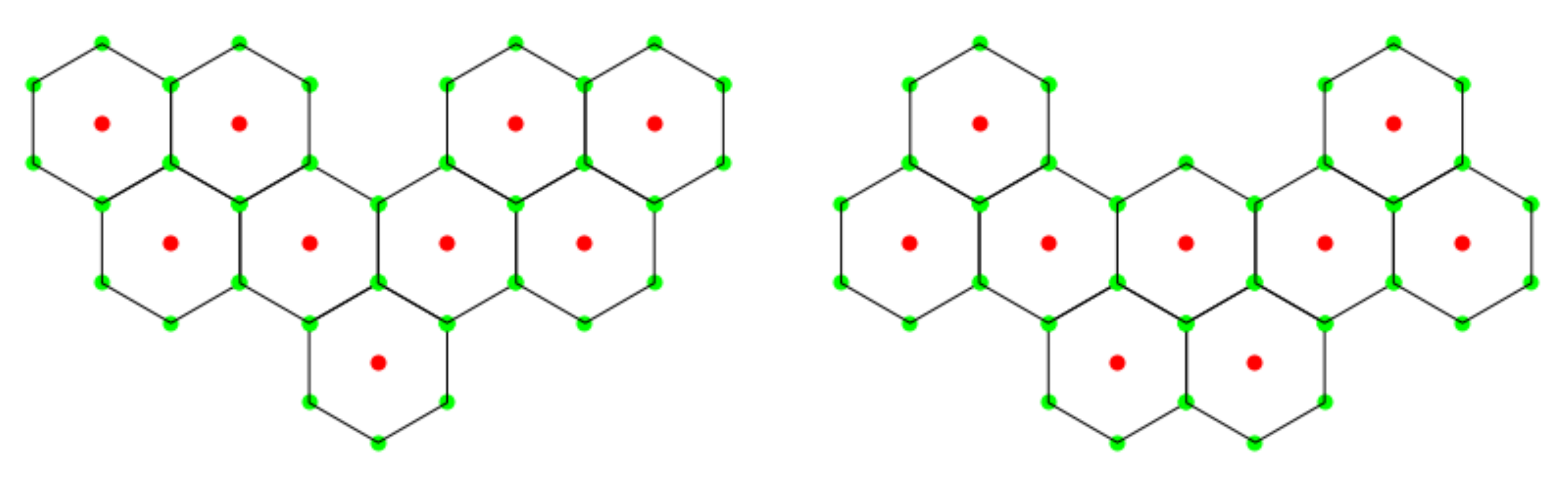}}
\scalebox{0.2}{\includegraphics{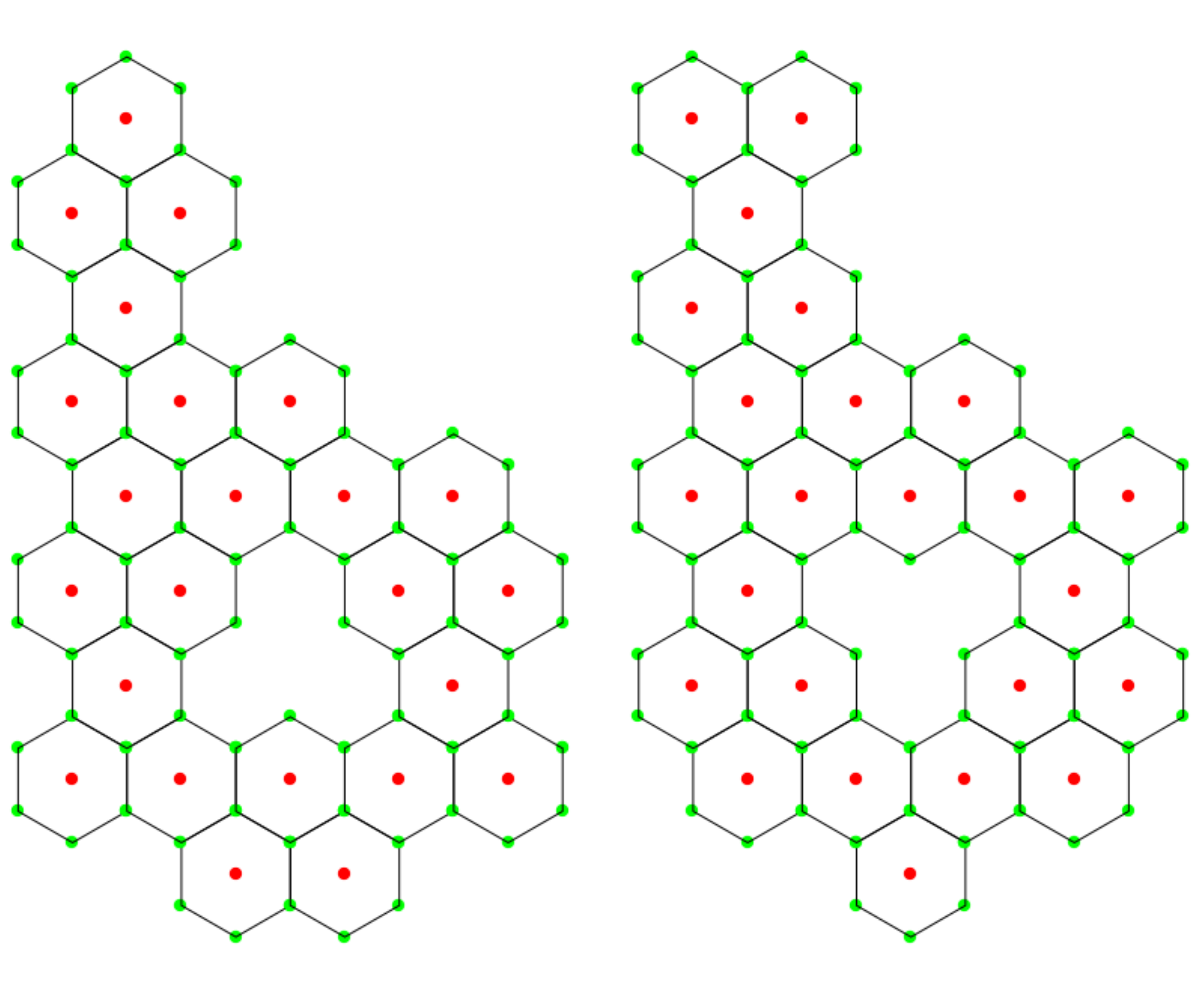}}
\caption{
Two examples of isospectral Benzenoid pairs which are isospectral 
for the adjacency matrix $A$. These graphs are one dimensional
since they contain no triangles. McKean-Singer shows that the 
nonzero spectrum of $L_0$ is the same than the nonzero spectrum 
of $L_1$. 
\label{benzenoid}
}
\end{figure}

\begin{figure}[ht]
\scalebox{0.35}{\includegraphics{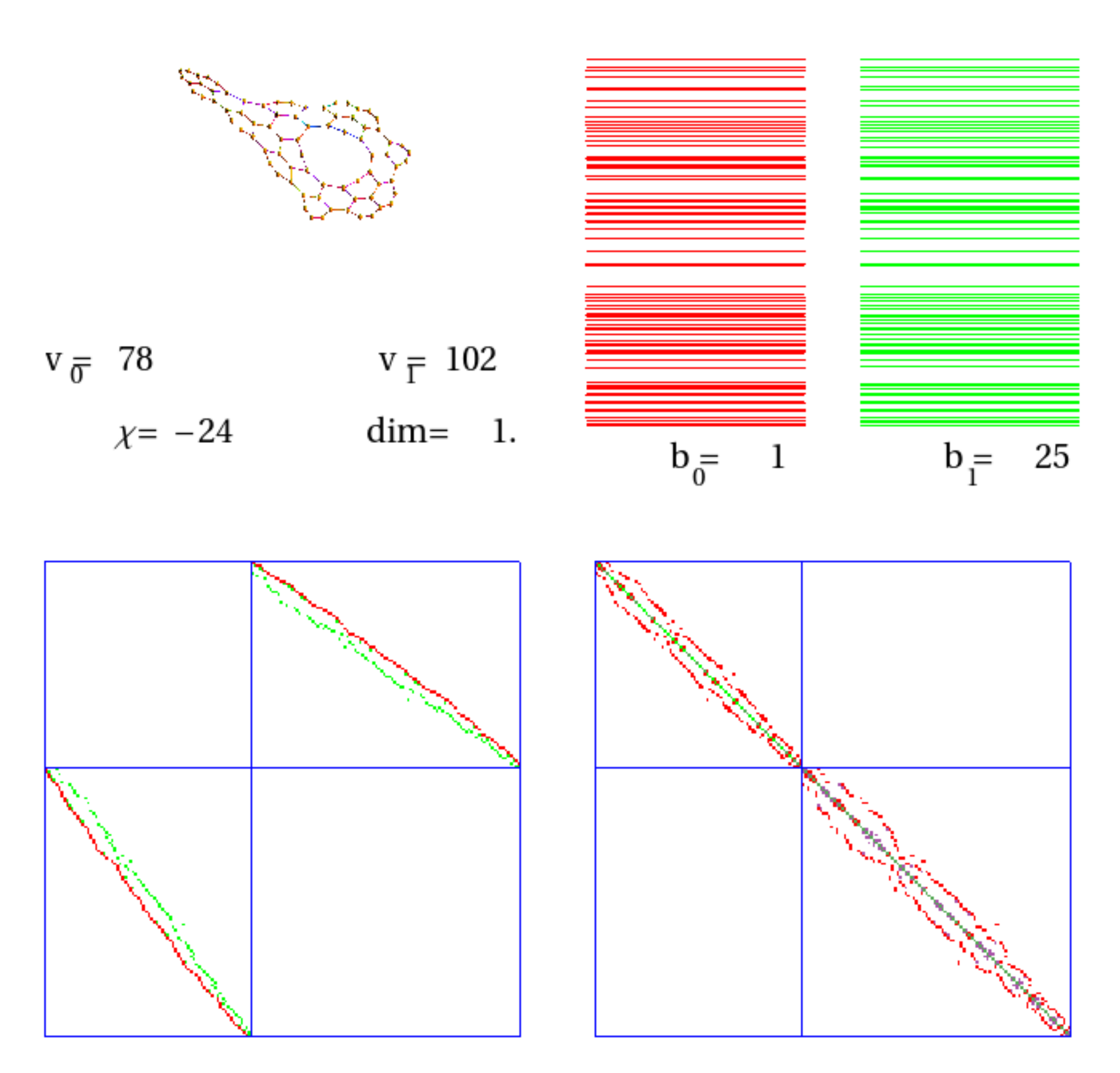}}
\caption{
The Dirac spectrum of the first Benzenoid. 
\label{benzenoid2}
}
\end{figure}

\vfill 
\pagebreak

\subsection{Gordon-Webb-Wolpert}

Isospectral domains in the plane were constructed with Sundada methods
by Gordon, Webb and Wolpert \cite{GWW}. 
A simplified version consists of two squares and three triangles. 

\begin{figure}[ht]
\scalebox{0.2}{\includegraphics{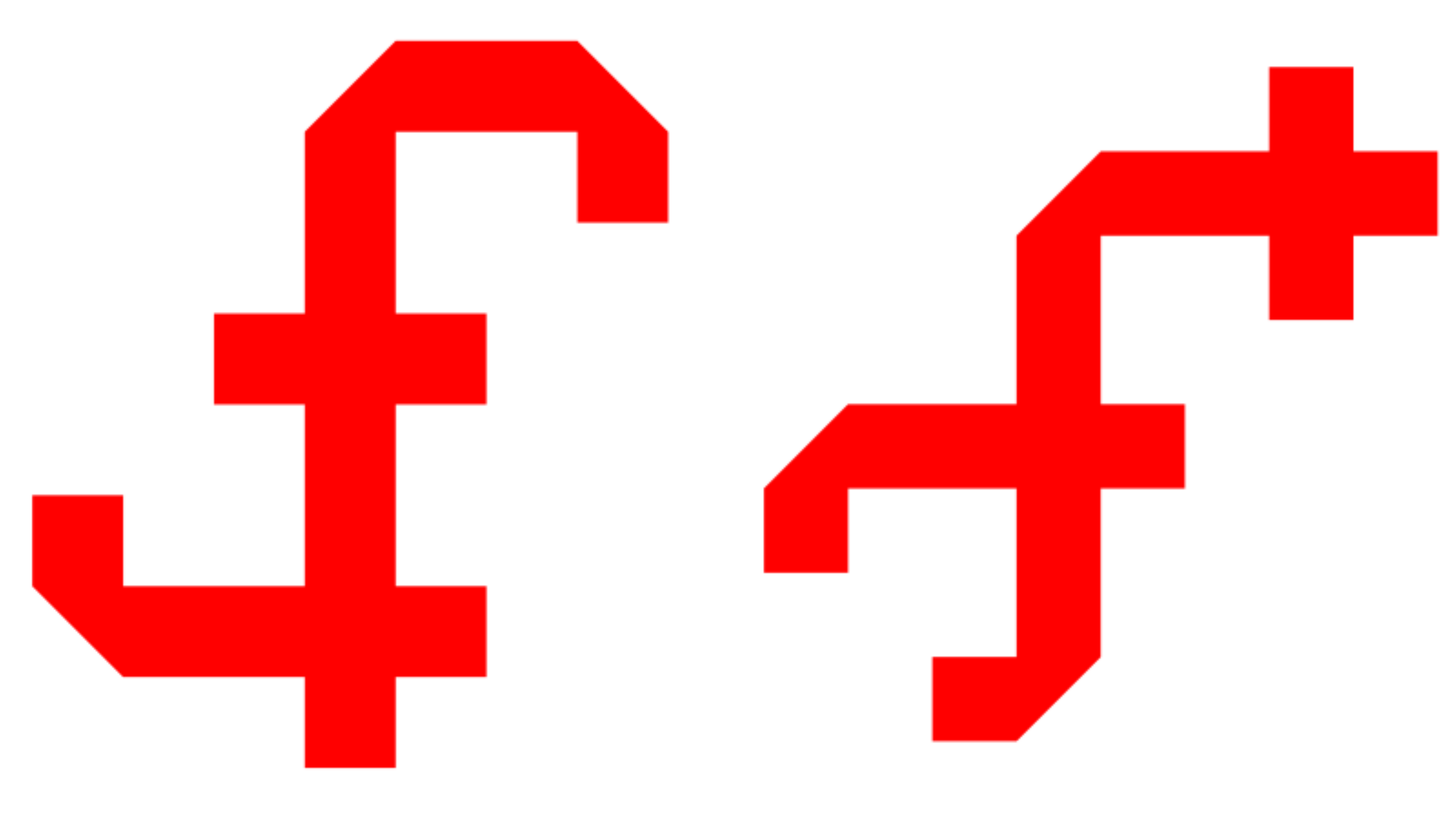}}
\scalebox{0.2}{\includegraphics{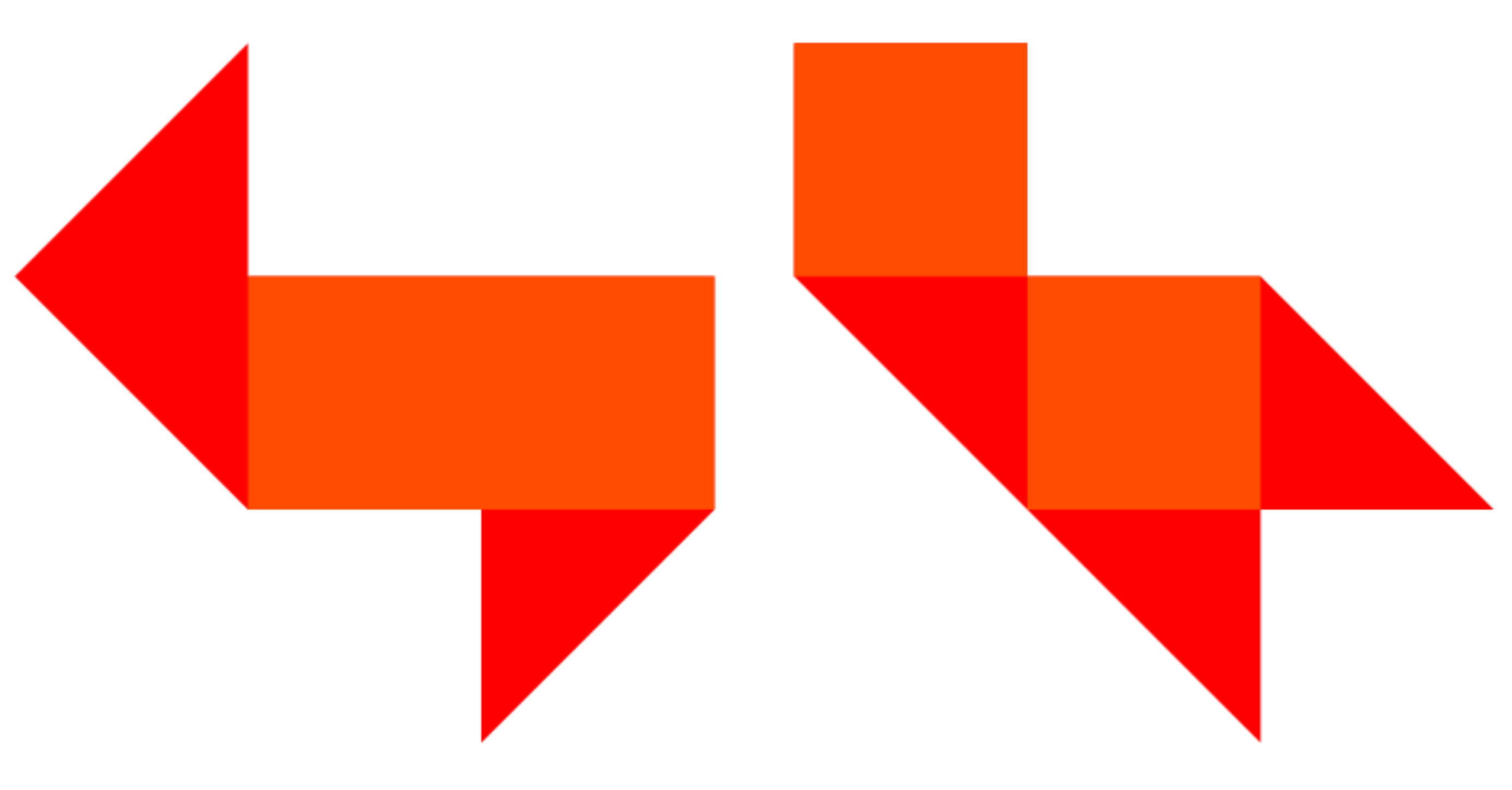}}
\caption{
\label{isospectraldrum}
}
\end{figure}

\begin{figure}[ht]
\scalebox{0.30}{\includegraphics{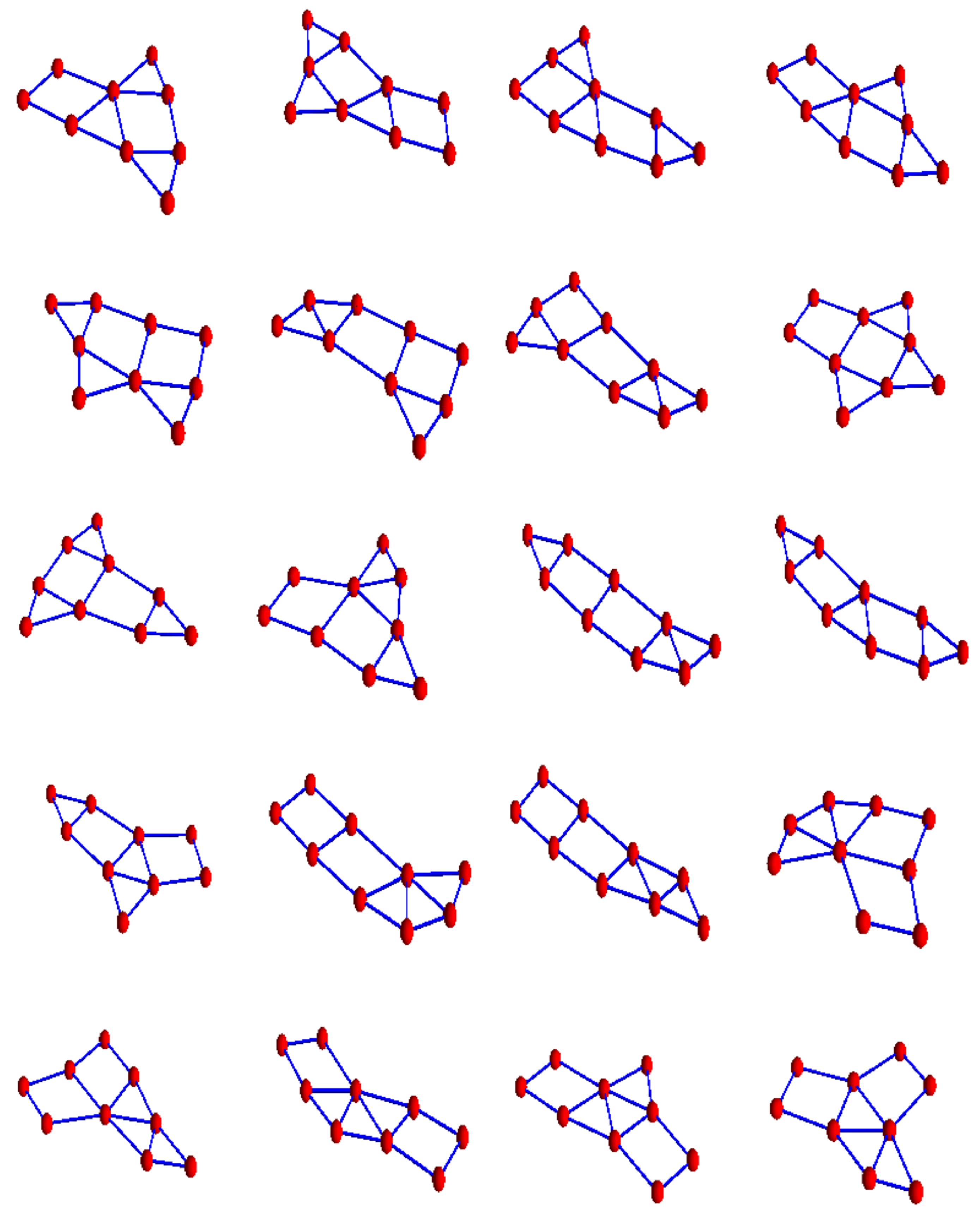}}
\caption{
\label{sunada}
}
\end{figure}

To try out graph versions of this, we computed the spectra 
of a couple of graphs of this type. They are all homotopic
to the figure 8 graph but are all not isospectral. 
This shows that discretizing Sunada can not be done naivly.
The next section mentions examples given in \cite{GWW} which 
were obtained by adapting Sunada type methods to graphs. 

\vfill
\pagebreak 

\subsection{Halbeisen-Hungerb\"uhler}

Here are the picture of the Dirac matrices and their spectrum of
the isospectral graphs constructed in \cite{GWW}. 

\begin{figure}[ht]
\scalebox{0.35}{\includegraphics{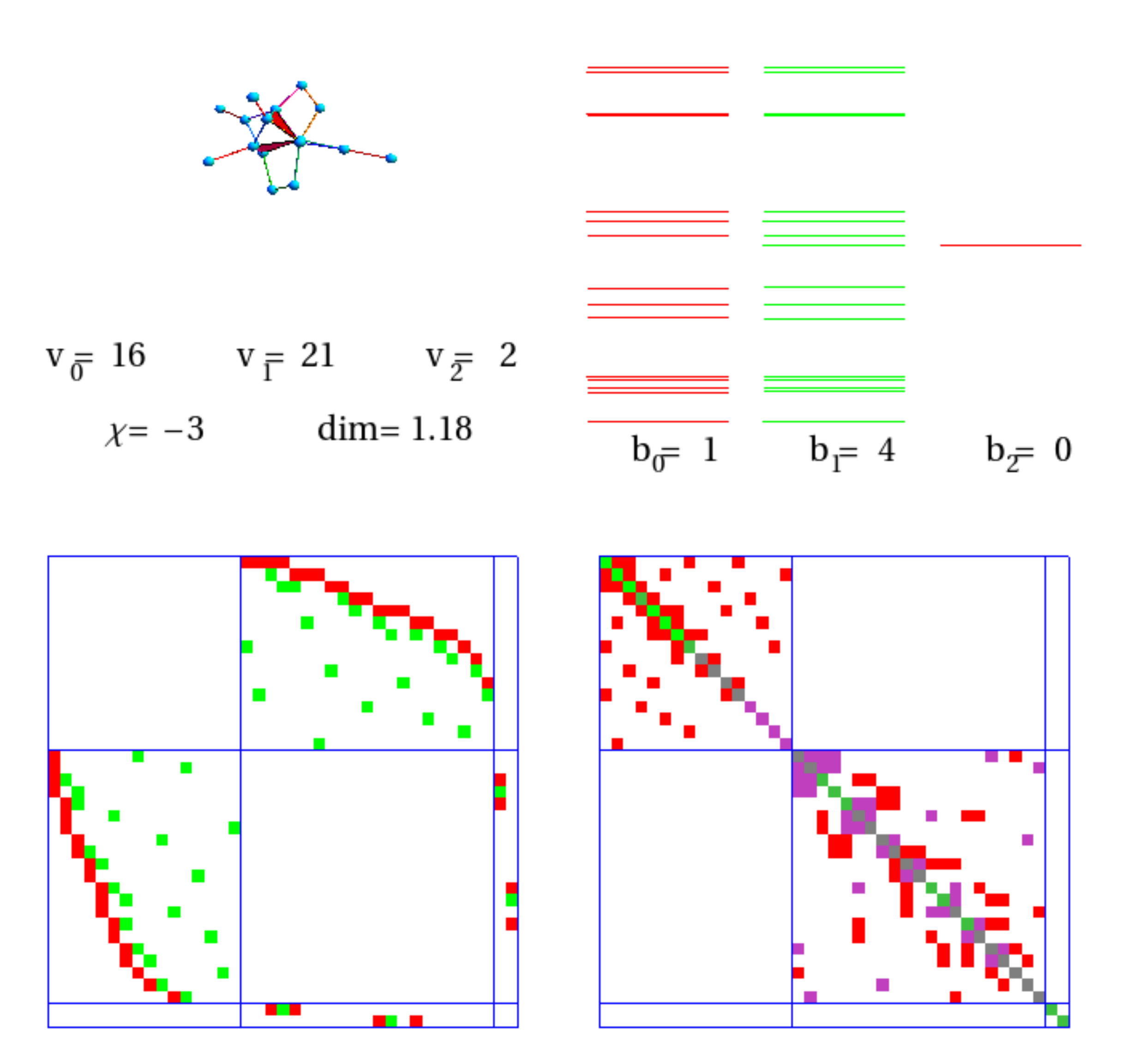}}
\scalebox{0.35}{\includegraphics{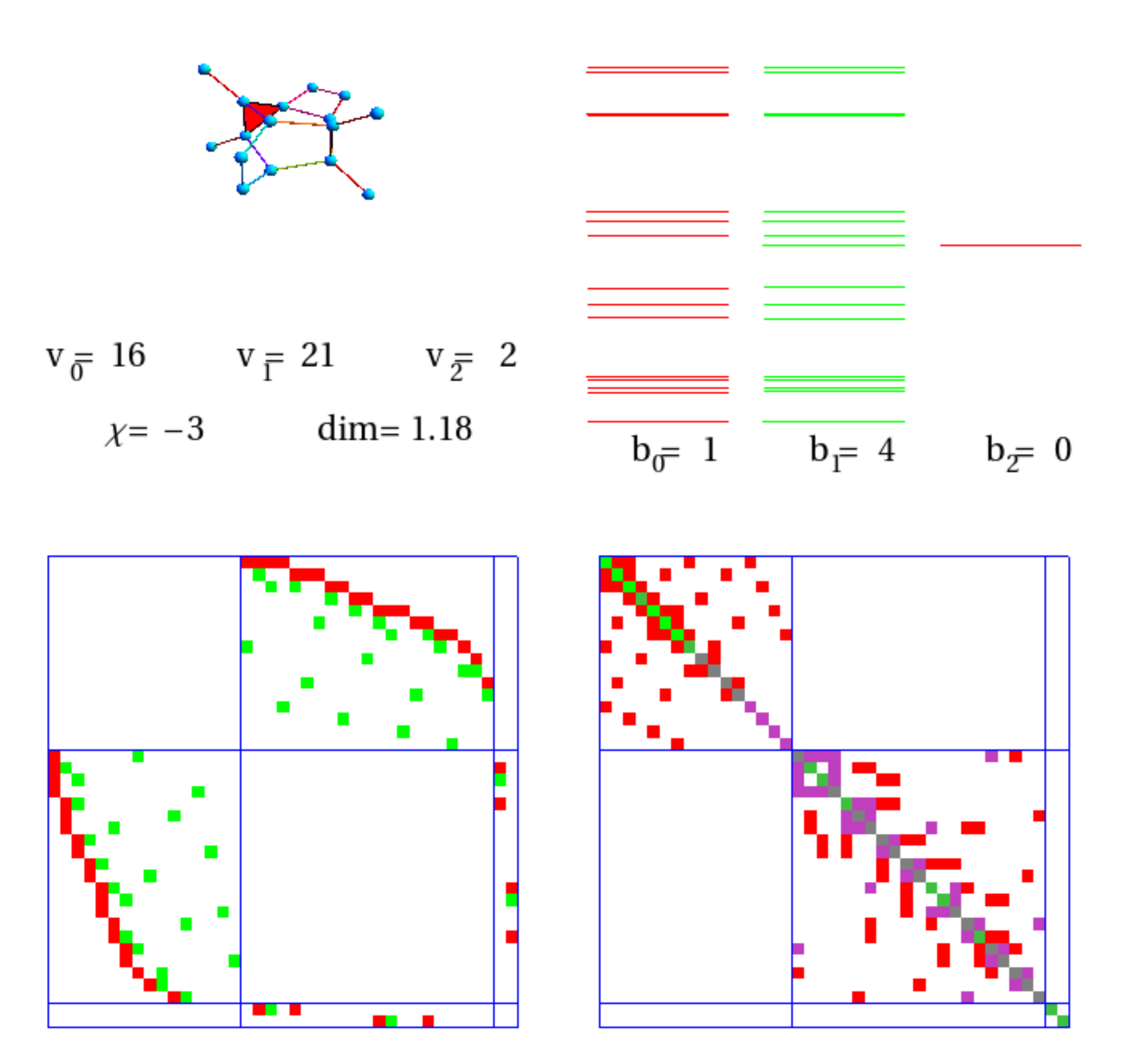}}
\caption{
\label{sunada}
}
\end{figure}

\vfill
\pagebreak

\subsection{Polyhedra}

For the 5 platonic solids, we have the following Dirac spectra and complexities $\prod_{\lambda \neq 0} \lambda$,
where we write $\lambda^n$ for $\lambda$ with multiplicity $n$.\\
\begin{tabular}{lrr} \hline
Solid       & Non-negative Dirac Spectrum          & Complexity \\ \hline
Tetrahedron &$\{2^7,0^1 \; \}$                                      & $-2^{14}$              \\
Octahedron  &$\{\sqrt{6}^3,\sqrt{2}^3,2^6,0^2 \;\}$                 & $2^{18} \cdot 3^3$     \\
Cube        &$\{\sqrt{6},2^3,\sqrt{2}^3,0^6 \;\}$                   & $-2^{10} \cdot 3$      \\
Dodecahedron&$\{\sqrt{2}^5,\sqrt{3}^4,\sqrt{5}^4,\sqrt{3-\sqrt{5}}^3,0^{12} \; \}$ & $-2^{11} \cdot 3^4 \cdot 5^4$     \\
Icosahedron &\begin{small}
$\{\sqrt{2}^5,\sqrt{3}^4,\sqrt{5}^4,\sqrt{6}^5,\sqrt{5\pm\sqrt{5}}^3,\sqrt{3\pm\sqrt{5}}^3,0^2 \; \}$
\end{small} & $2^{22} \cdot 3^9 \cdot 5^7$    \\ \hline
\end{tabular}

The following table confirms that for 
two-dimensional geometric graph for which the unit sphere is a circular graph, the 
Dirac complexity is positive. We actually came up with the statement after watching the 
following tables: 

\begin{small}
$$ A = \left[
                 \begin{array}{rrr}
                  {\rm dim}(G) & {\rm complexity} & \chi(G) \\ \hline
                  2   & -2.6121\times 10^{10} & -4 \\
                  1   & -2.6147\times 10^{42} & -60 \\
                  1   & -5.9608\times 10^{17} & -24 \\
                  2   & -2.1859\times 10^{25} & -10 \\
                  3/2 & -4.2166\times 10^{40} & -40 \\
                  3/2 & -4.7406\times 10^{16} & -16 \\
                  2   & -2.1456\times 10^{28} & -4 \\
                  2   & -8.2030\times 10^{69} & -10 \\
                  5/3 & -5.1018\times 10^{12} & -4 \\
                  5/3 & -1.0423\times 10^{30} & -10 \\
                  1   & -2.2518\times 10^{22} & -30 \\
                  1   & -2.4277\times 10^9 & -12 \\
                  \frac{5}{3} & -5.8320\times 10^6 & -2 \\
                 \end{array}
                 \right] \; , 
   C = \left[
                 \begin{array}{rrr}
                  {\rm dim}(G) & {\rm complexity} & \chi(G) \\ \hline
                   1 & -4.6449\times 10^6 & -10 \\
                   2 & 3.5323\times 10^{84} & 2 \\
                   2 & 1.9246\times 10^{35} & 2 \\
                   1 & -6.6871\times 10^{15} & -28 \\
                   1 & -1.2496\times 10^{31} & -58 \\
                   1 & -7.8275\times 10^{12} & -22 \\
                   1 & -3.4164\times 10^{15} & -22 \\
                   1 & -4.0315\times 10^{37} & -58 \\
                   3 & 3.7873\times 10^{26} & 2 \\
                   3 & 1.4404\times 10^{64} & 2 \\
                   2 & 2.7042\times 10^{44} & 2 \\
                   2 & 3.4381\times 10^{18} & 2 \\
                   3 & -1.6987\times 10^{14} & 1 \\
                  \end{array}
                 \right] \; . $$
\end{small}
\begin{figure}[ht]
\scalebox{0.25}{\includegraphics{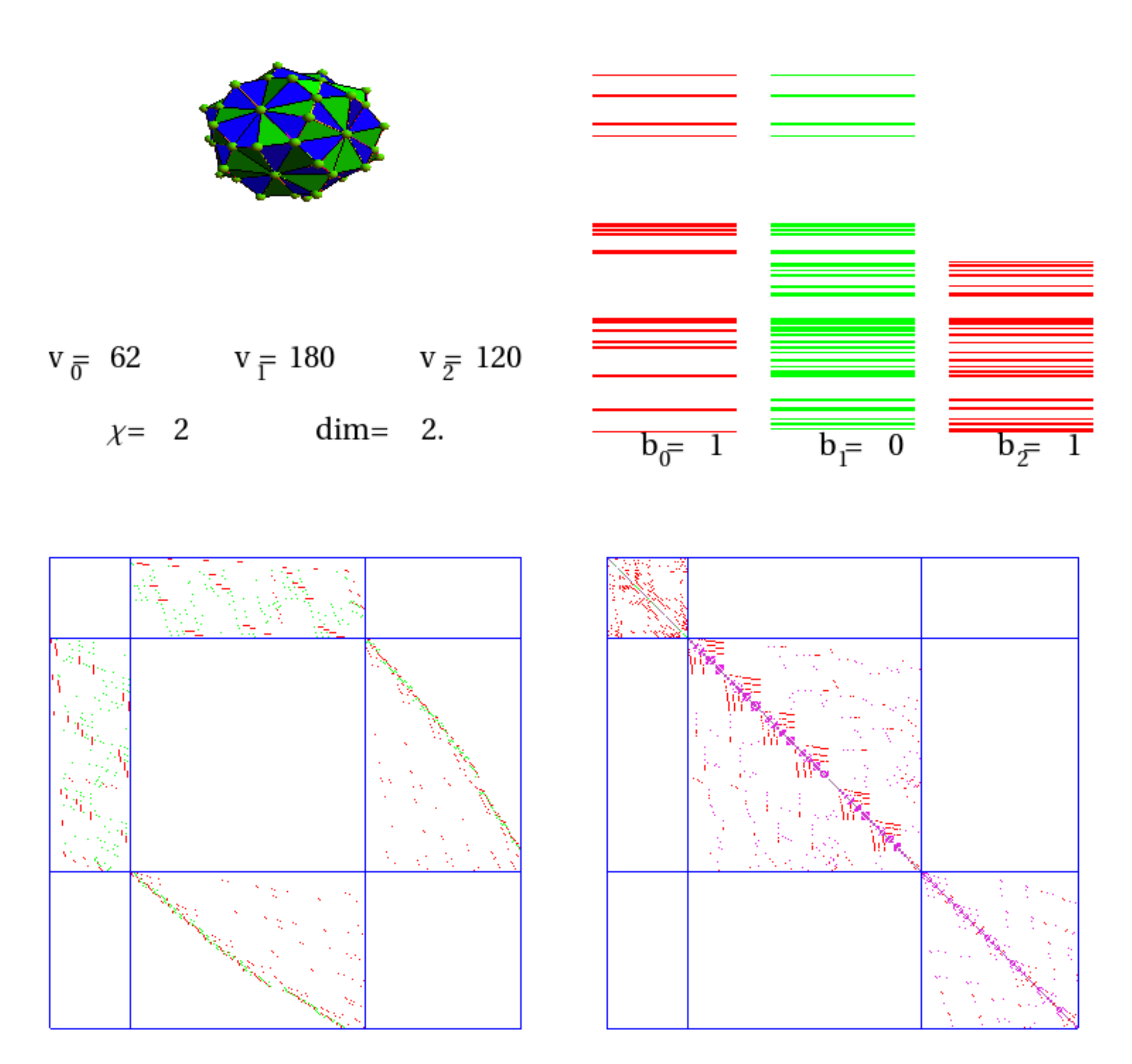}}
\caption{
The DisdyakisTriacontahedron is the Catalan solid with maximal complexity.
\label{DisdyakisTriacontahedron}
}
\end{figure}

\vfill
\pagebreak

\section*{Appendix: The original proof}

While the super symmetry argument is elegant, the original proof
is illustrative and shows better the link to Hodge theory. 
We follow \cite{McKeanSinger} but change notation and expand some of the steps slightly. 
Denote by $E_{\lambda}^p$ the eigenspace of the eigenvalue $\lambda$ of $L$ on $\Omega^p(G)$.

\begin{propo}[McKean-Singer]
$\sum_p (-1)^p \dim(E_{\lambda}^p)=0$ for all $p \geq 0$.
\label{mckeansingerpropo}
\end{propo}

\begin{proof}
If $\lambda$ is a positive eigenvalue, we just write $E^p$ instead of $E^p_{\lambda}$. \\
(i) $E^p = d E^{p-1} \oplus d^* E^{p+1}$ \\
Proof. That the right hand side is included in the left hand side follows
from $[d,L]=[d^*,L]=0$. For example: if $f=dg$ with $L g = \lambda g$, then 
$Lf = L dg = d Lg = d \lambda g = \lambda dg = \lambda f$ and $f \in E^p$. \\
Given $f \in d E^{p-1},g \in d^* E^{p+1}$, then $\langle f,g \rangle = \langle dh,d^* k \rangle = \langle d^2 h,k \rangle=0$
so that the two linear spaces on the right hand side are perpendicular. Assume we are
given $f \in E^p$ which is perpendicular to both subspaces. Then 
$0=\langle f,d E^{p-1} \rangle=\langle d^*f,E^{p-1} \rangle$
and $0=\langle f,d^* E^{p+1} \rangle=\langle df,E^{p+1} \rangle$ 
so that $d^*f=df=0$ and $Lf=0$ which together with $f \in E^p$ implies $f=0$.  \\
(ii) Summing the identity $\dim(E^p) = \dim(d E^{p-1}) + \dim(d^* E^{p+1})$ obtained in (i) gives
$\sum_p (-1)^p\dim(E^p)=$
\begin{eqnarray*}
   &=& \sum_{p \; even}  \dim(E^p) - \sum_{p odd}  (\dim(d E^{p-1}) + \dim(d^* E^{p+1}) )   \\
   &=& \sum_{p \; even}  \dim(E^p) - \dim(d E^p) - \dim(d^* E^p)  \\
   &=& \sum_{p \; even}  \dim(d E^{p-1}) + \dim(d^* E^{p+1}) - \dim(d E^p) - \dim(d^* E^p) \\
   &=& \sum_{p \; even}  \dim(d E^{p-1}) + \dim(d^* E^{p+1}) - \dim(d d^* E^{p+1}) - \dim(d^* d E^{p-1}) \\
   &=& \sum_{p \; even} [\dim(d E^{p-1})-\dim(d d^* E^{p+1})]+[\dim(d^* E^{p+1})-\dim(d^* d E^{p-1})] \geq 0 \; . 
\end{eqnarray*}
(iii) $\sum_p (-1)^p \dim(E^p) \leq 0$ \\
From (ii), we recycle $\sum_p (-1)^p\dim(E^p)=\sum_{p \; even}  \dim(E^p) - \dim(d E^p) - \dim(d^* E^p)$.  
The claim follows now because each term is $\leq 0$: 
\begin{eqnarray*}
   & &   \dim(E^p) - \dim(d E^p) - \dim(d^* E^p)  \\ 
   &=&   \dim(L E^p) - \dim(d E^p) - \dim(d^* E^p) \\
   &\leq& \dim(d d^* E^p) + \dim(d^* d E^p) - \dim(d E^p) - \dim(d^* E^p) \\
   &=&    [ \dim(d d^* E^p) - \dim(d E^p)] + [ \dim(d^* d E^p) - \dim(d^* E^p) ] \leq 0  \; . 
\end{eqnarray*}
The two parts (ii) and (iii) imply the proposition. 
\end{proof}

\begin{coro}
$\str(e^{-tL}) = \chi(G)$.
\end{coro}
\begin{proof}
Proposition~(\ref{mckeansingerpropo}) implies $\str(L^k)=0$ for $k>0$. 
A Taylor expansion of the super trace of the heat kernel using the proposition gives
$$ \str(e^{-tL}) = \str(1-t \frac{L}{1!} + t^2 \frac{L^2}{2!} - \cdots) = \str(1) = \chi(G) \; $$
\end{proof} 

\section*{Appendix: Hodge theory for graphs}

For $t=0$ we have by definition the Euler characteristic and for $t \to \infty$, the 
supertrace of the identity on harmonic forms. The proof of the Hodge theorem 
stating that $H^p(G)$ is isomorphic to the space $E_0^p(G)$ of harmonic $p$-forms
is clearly visible in McKean-Singers proof: 

\begin{lemma}
a) $d^* L = L d^* = d^* d d^*$ and $d L = L d = d d^* d$. \\
b) $L f = 0$ is equivalent to $df=0$ and $d^* f = 0$. \\
c) $f=dg$ and $h=d^*k$ are orthogonal. ${\rm im}(d) \cup {\rm im}(d^*)$ span $im(L)$.
\end{lemma}

\begin{proof}
a) is clear from $d^2=(d^*)^2=0$.\\
b) if $Lf=0$ then $0=\langle f,Lf \rangle=\langle d^*f,d^f \rangle + \langle df,df \rangle$ 
shows that both $d^* f=0$ and $df=0$. The other direction is clear. \\
c) $\langle f,h \rangle = \langle dg,d^*k \rangle 
   = \langle d dg,k \rangle = \langle g,d^* d^* k \rangle=0$.
\end{proof}

Any $p$-form can be written as a sum of an exact, a coexact and a harmonic $p$-form.

\begin{lemma}[Hodge decomposition]
There is an orthogonal decomposition 
$$ \Omega = {\rm im}(L) + {\rm ker}(L) = {\rm im}(d) + {\rm im}(d^*) + {\rm ker}(L)  \; . $$ 
Any $g$ can be written as $g= df + d^* h + k$ where $k$ is harmonic.
\end{lemma}

\begin{proof}
The operator $L: \Omega \to \lambda$ is symmetric so that image and kernel are
perpendicular. We have seen before that the image of $L$ splits into two
orthogonal components ${\rm im}(d)$ and ${\rm im}(d^*)$.
\end{proof}

\begin{thm}[Hodge-Weyl]
The dimension of the vector space of harmonic $k$-forms on
a simple graph is $b_k$. Every cohomology class has a unique
harmonic representative.
\end{thm}

\begin{proof}
If $Lf=0$, then $df=0$ and so $L f = d^* d f=0$. Given a closed $n$-form $g$
then $dg=0$ and the Hodge decomposition shows $g=df + k$ so that $g$
differs from a harmonic form only by $df$ and so that this harmonic form
is in the same cohomology class than $f$. We can assign so to a cohomology
class a harmonic form and this map is an isomorphism.
\end{proof}

\section*{Appendix: \c{C}ech cohomology for graphs}

We have mentioned that the Dirac operator has appeared when computing \v{C}ech cohomology
of a manifold and investigating the relation between the spectrum of the graph and the 
spectrum of the manifold \cite{Mantuano}. We want to show here that \v{C}ech cohomology
is equivalent to graph cohomology. This is practical: 
while graph cohomology is easy to implement in a computer, it can be tedious to compute
the operator $D$ and so cohomology groups. If the zero eigenvalues and so the cohomology 
is the only interest in $D$, then one can proceed in two different but related ways:
One possibilities to reduce the complexity of computing the kernel of $D$ is to deform 
the graph using homotopy steps to get a simpler graph. Another possibility is to look at 
the Dirac operator of the nerve. The following definition is equivalent to a definition of 
Ivashchenko:

\begin{defn}
A graph is called {\bf contractible in itself}, if there is a proper subgraph $H$ which is
contractible and which is the unit sphere of a vertex. The graph $G$ without the vertex $x$ and
all connections to $x$ is called a contraction of $G$. For a contractible graph, there is 
a sequence of such contraction steps which reduces $G$ to a one point graph. 
\end{defn}

This allows now to consider \v{C}ech cohomology for graphs: 

\begin{defn}
A \c{C}ech cover of $G$ is a finite set of subgraphs $U_j$ which are contractible and for
which any finite intersection of such subgraphs is either empty or contractible. A \v{C}hech
cover defines a new graph $N$ called the {\bf nerve} of the cover. The vertices of $N$ are 
the elements of the cover and there is an edge $(a,b)$ if both the subgraphs $a,b$ have a nonempty 
intersection. {\bf \v{C}ech cohomology} is then defined as the graph cohomology of the nerve of $G$. 
\end{defn}

{\bf Remark.} The number of vertices of the nerve of $G$ is an upper bound for the geometric category 
$\gcat(G)$ of $G$ and so to strong category $\Cat(G)$ \cite{josellisknill}. \\

\begin{propo}
The nerve of a \c{C}ech cover of $G$ is homotopic to $G$. 
\end{propo}
\begin{proof}
In a first step  we expand each of the $U_j$ so that in each $U_j$, 
the intersections $U_k \cap U_j$ are in a 
neighborhood $B(x_j)$ of a point $x_j \in U_j$. 
Now homotopically shrink each $U_j$ to $B(x_j)$. 
\end{proof}

\begin{coro}
\c{C}ech cohomology on a finite simple graph is equivalent to graph cohomology on that graph.
\end{coro}

\begin{proof}
Graph cohomology obviously is identical to \v{C}ech cohomology if
we chose the cover $U_j = B(x_j)$ for the vertices $x_j$ which has the
property that the nerve of $G$ is equal to $G$. Because the nerve $N$
of $G$ is homotopic to $G$ and graph cohomology of homotopic graphs is
the same (as already proven by Ivashchenko \cite{I94}), we are done. 
\end{proof}

\vfill 
\pagebreak

\vspace{12pt}
%\bibliographystyle{plain}
%\bibliography{geometry}

\end{document}